\documentclass[reqno,11pt]{amsart}
\usepackage{amsmath, amsfonts,amsthm,amssymb,amsbsy,upref,color,graphicx,amscd,hyperref,enumerate}
\usepackage[active]{srcltx}
\usepackage[latin1]{inputenc}
\usepackage{bbm}
\usepackage{pgf}
\usepackage{xfrac}

\def\ds{\displaystyle}
\def\eps{{\varepsilon}}
\def\N{\mathbb{N}}

\def\R{\mathbb{R}}

\def\HH{\mathcal{H}}

\def\KK{\mathcal{K}}

\def\W{\widetilde W}

\newcommand{\be}{\begin{equation}}
\newcommand{\ee}{\end{equation}}

\newcommand{\de}{\partial}
\newcommand{\dist}{{\rm {dist}}}

\newcommand{\reg}{{\rm Reg}}
\newcommand{\sing}{{\rm Sing}}

\theoremstyle{plain}
\newtheorem{theo}{Theorem}

\newtheorem{exam}{Example}

\numberwithin{equation}{section}
\theoremstyle{plain}
\newtheorem{teo}{Theorem}[section]
\newtheorem{lemma}[teo]{Lemma}

\newtheorem{prop}[teo]{Proposition}

\theoremstyle{remark}
\newtheorem{oss}[teo]{Remark}

\title[A logarithmic epiperimetric inequality for the obstacle problem]{A logarithmic epiperimetric inequality for the obstacle problem}

\author{Maria Colombo, Luca Spolaor, Bozhidar Velichkov}

\address {Maria Colombo: \newline \indent
Institute for Theoretical Studies, ETH Z\"urich,
	\newline \indent
 Clausiusstrasse 47, CH-8092 Z\"urich, Switzerland
 	}
\email{maria.colombo@eth-its.ethz.ch}

\address {Luca Spolaor: \newline \indent
	Massachusetts Institute of Technology (MIT), 
	\newline \indent
	77 Massachusetts Avenue, Cambridge 
	MA 02139, USA}
\email{lspolaor@mit.edu}

\address {Bozhidar Velichkov: \newline \indent
Laboratoire Jean Kuntzmann (LJK), Universit\'e Grenoble Alpes
\newline \indent
B\^atiment IMAG, 700 Avenue Centrale, 38401 Saint-Martin-d'H\`eres}
\email{bozhidar.velichkov@univ-grenoble-alpes.fr}

\makeindex

\begin{document}


\begin{abstract}
For the general obstacle problem, we prove by direct methods an epiperimetric inequality at regular and singular points, thus answering a question of Weiss (Invent.
Math., 138 (1999), 23--50). In particular at singular points we introduce a new tool, which we call logarithmic epiperimetric inequality, which yields an explicit logarithmic modulus of continuity on the $C^1$ regularity of the singular set, thus improving previous results of Caffarelli and Monneau \cite{Caff3,Caff_rev,Monneau}.
\end{abstract}

\maketitle



\section{Introduction}	

In this paper we study the regularity of the free-boundary of nonnegative local minimizers $u$ of the functional 
$$
\mathcal E(u):=\int |\nabla u|^2\,dx+\int \max\{u(x),0\}\,dx\,.
$$
Our main result is a logarithmic epiperimetric inequality, which is a new tool for the study of the singular set of minimizers of variational energies. It is also an alternative approach to the  regularity of the singular free boundary as proposed by Caffarelli \cite{Caff3,Caff_rev}. Before we state it we recall that, given $u\in H^1(B_1)$, the \emph{Weiss' boundary adjusted energy of $u$} is defined by
$$
W(u)=\int_{B_1}|\nabla u|^2\,dx-2\int_{\partial B_1}u^2\,d\HH^{d-1}+\int_{B_1}\max\{u(x),0\}\,dx\,.
$$
The class $\KK$ of admissible blow-ups of $u$ at singular points is defined by
\begin{equation}
\label{defn:K}
\KK:=\{Q_A \colon \R^d \to \R \,:\, Q_A(x) = x \cdot Ax,\,\text{ $A$  symmetric non-negative with }{\rm{tr}} A = \sfrac14\}\,. 
\end{equation}
{The energy $W$ is constant on $\KK$, precisely we have $ W(Q_A)=\frac{\omega_d}{8\, d (d+2)}$, for every $Q_A\in \KK$. We refer to this constant as to the \emph{energy density at the singular points} and denote it by $\Theta$.}

\begin{theo}[Logarithmic epiperimetric inequality at singular points]\label{t:epi:B}
There are dimensional constants $\delta>0$ and $\eps>0$ such that the following claim holds. For every non-negative function $c\in H^1(\partial B_1)$, with $2$-homogeneous extension $z$ on $B_1$, satisfying   
$$\text{dist}_{L^2(\partial B_1)}\left(c,\KK\right)\le \delta\qquad\text{and}\qquad 0\le W(z)-\Theta\le 1,$$
{\flushright there is a non-negative function $h\in H^1(B_1)$ with $h=c$ on $\partial B_1$ satisfying the inequality }
\begin{equation}\label{e:epi_flat_point_sing}
W(h)\le W(z)-\eps\big(W(z)-\Theta\big)^{1+\gamma}\,,\qquad \mbox{where}\quad
\begin{cases}
\gamma=0 & \mbox{if }d=2\\
\gamma=\frac{d-1}{d+3}  & \mbox{if }d\geq3
\end{cases}.
\end{equation}
\end{theo}

At flat points we recover the Weiss' epiperimetric inequality with a direct proof. To state it, recall that the collection $\KK_+$ of possible blow-ups at flat points is defined by
\begin{equation}
\label{defn:K+}
\KK_+:=\left\{q_\nu\colon \R^d\to \R\,:\,q_\nu(x)=(\max\{x\cdot\nu,0\})^2\mbox{ for some } \nu\in\R^d\text{ such that } |\nu|=\frac12\right\}.
\end{equation}
{The energy $W$ is constant on $\KK_+$, precisely we have $ W(q_\nu)=\frac{\omega_d}{16\, d (d+2)}$, for every $q_\nu\in \KK_+$. We will refer to this constant as the \emph{energy density at the flat points} and denote it by $\Theta_+$.}

\begin{theo}[Epiperimetric inequality at flat points]\label{t:epi:A}
There are dimensional constants $\delta_0>0$, $\delta>0$ and $\eps>0$ such that the following claim holds. For every non-negative function $c\in H^1(\partial B_1)$ satisfying
$$
\{x_d<-\delta_0\}\subset\{c=0\} \quad \mbox{and}\quad \|c-q_{e_d}\|_{L^2(\partial B_1)}\le \delta\,,
$$ 
there exists a non-negative function $h\in H^1(B_1)$ such that $h=c$ on $\partial B_1$ and 
	\begin{equation}\label{e:epi_flat_point}
	W(h)-\Theta_{+}\le (1-\eps)\big(W(z)-\Theta_{+}\big)\,,
	\end{equation}
	where $z$ is the $2$-homogeneous extension of $c$ to $B_1$.
\end{theo}

Theorem \ref{t:epi:A} was already proved by Weiss in \cite{Weiss2} using a very elegant and innovative contradiction argument, later exploited also by Garofalo-Petrosyan-Garcia and Focardi-Spadaro in the context of the thin obstacle problem (see \cite{GaPeVe,FoSp}). { However, the same proof works only at singular points of maximal and minimal dimension under some special assumptions on the projection of the trace on $\KK$, which can be verified only in dimension $d=2$. Notice that the dimension of a singular point is the maximal $\dim(\ker A)$ among all $Q_A\in \KK$ blow-ups of $u$ at the singular point.  Hence, no epiperimetric inequality was known in the literature for the whole singular set, as it happens in Theorem~\ref{t:epi:B}} and Weiss himself suggests that  "...it should however be possible to give a direct proof of the epiperimetric inequality which would then also cover singular sets of intermediate dimension" (see \cite{Weiss2}). Theorems \ref{t:epi:B} and \ref{t:epi:A} answer affirmatively to this question, and in particular Theorem \ref{t:epi:B} is the first instance in the literature of an epiperimetric inequality of logarithmic type and the first instance in which the epiperimetric inequality for singular points has a direct proof. The methods developed to prove the epiperimetric inequality at singular points of any stratum have a quite general nature and will be applied to provide similar results in other problems, for instance in the case of the thin obstacle problem \cite{CSVthin}.

The proof of these theorems is direct (i.e. we produce explicit competitors) and it is remarkable in our opinion how the failure of Weiss' contradiction argument translates into a weakening of the epiperimetric inequality, that is the necessity of introducing the exponent $\gamma$ in \eqref{e:epi_flat_point_sing}. To explain this better, notice that, in analogy with Reifenberg and White's pioneering work (see \cite{Reif2,Wh}) and similarly to previous work of the last two authors (see \cite{SpVe}), the key ingredients are
\begin{itemize}
	\item a Fourier decomposition of the trace $c-q_\nu$ (resp. $c-Q_A$) onto the eigenfunctions of $\mathbb{S}^{d-1}$; 
	\item an energy improvement with respect to $z$ obtained by taking the harmonic extension of the modes with homogeneity larger than two;
	\item a choice of $\nu$  (resp. $A$) to control the projection of $c-q_\nu$ (resp. $c-Q_A$) onto the eigenfunctions of homogeneity one and two, with the projection on the higher modes.
\end{itemize}
At flat points and at every point in dimension $d=2$, the estimate of the third bullet is linear, however in general dimension at singular points we can only prove a control of the form
$$
\|P (c-Q_A)\|_{H^1(\de B_1)}\le C\|(Id-P)(c-Q_A)\|_{H^1(\partial B_1)}^{1-\gamma}\qquad \gamma\in(0,1)\,,
$$
where $P$ denotes the projection on the modes relative to homogeneity two (see \eqref{e:sing:pre-final}). The reason for this different behavior is essentially the following: {at the flat points we are able to eliminate the lower modes (the modes corresponding to homogeneity smaller than two) on a spherical cap by means of the choice of the vector $\nu$; this is possible since the space of admissible functions $q_\nu$ is an open manifold of the same dimension as the eigenspace corresponding to the lower modes, so we can apply an implicit function argument (see Lemma \ref{l:inversa}).} 
At general singular points, we would like to eliminate the modes of homogeneity two, that is the modes corresponding to the eigenvalue $2d$ on the sphere and whose eigenspace can be identified with the space of $d\times d$ real symmetric matrices $S_d(\R)$. However, the positivity constraint on the competitor forces the choice of the matrix $A$ to be in the space of nonnegative symmetric matrices $S_d^+(\R)$.
Now these two spaces have the same dimension, but, due to the non-negativity assumption, the set $S_d^+(\R)\subset S_d(\R)$ is not open, so we cannot apply the implicit function theorem here. Indeed, if we are in its interior, which corresponds to the singular points studied by Weiss, then the argument works and we can eliminate the second modes; but at the boundary of $S_d^+(\R)$ an implicit function argument only provides us with a matrix in the larger space $S_d(\R)$.
This leaves us to estimate the difference between the element of $S_d(\R)$, corresponding to the second modes of the trace $c$, and its projection on $S_d^+(\R)$. We can do this by means of the additional condition that $c$ is positive, which suggests that this difference should be comparable to the higher modes of the trace, but because of capacitary reasons the bound comes with an exponent $\gamma\neq 0$. Roughly speaking, if the negative part produced by the second modes is very small, concentrated on a set of small capacity, then it can be compensated by a function with very small energy, much smaller than the distance to $S_d^+(\R)$ in the space of symmetric matrices. In particular, it seems that this obstruction is of the same nature as the one that appears in \cite{Weiss2}, where the strong convergence of the traces cannot \emph{see} the nodal sets of small capacity.

\noindent A similar phenomenon can be found in the theory of minimal surfaces. Indeed if we think about the collection of singular points of lower dimension as a minimal surface in codimension higher than one, then it is known the existence of non-integrable cones, that is cones with non-integrable Jacobi fields. In this case the best possible rate of convergence to the blow-up is indeed logarithmic, as shown in \cite{AdSi}.

\noindent We should remark that an estimate of the form \eqref{e:epi_flat_point_sing} is essentially the best one can get by using only the positivity of the trace $c$ (see Example \ref{esempio}). It follows that Theorem \ref{t:epi:B} is essentially optimal, as it concerns positive traces on the sphere. However it is conceivable that for solutions of the obstacle problem a better inequality could be obtained, by using more properties of the minimizers.

It is well known that Theorem \ref{t:epi:A} leads to the uniqueness of the blow-up at every flat point and also to the $C^{1,\alpha}$ regularity of the regular part of the free-boundary (see \cite{Weiss2}). We show that Theorem \ref{t:epi:B} yields the uniqueness of the blow-up and the $C^{1}$ regularity of the singular set, with an explicit logarithmic modulus of continuity. This is an improvement on the results of Caffarelli and Monneau, where such a modulus arises by contradiction arguments and is therefore not explicit (see \cite{Caff3,Monneau}). The method of the present paper is flexible enough to cover more general and nonlinear functionals, such as the area.  The stratification of the singular set for the area functional, even in the context of Riemannian manifolds, and the $C^1$ regularity of the strata were recently obtained in \cite{FoGeSp,FoGelliSp}. Before giving the precise statements, we need some additional definitions.
We split the free-boundary of a minimizer $u$ in \emph{regular} and \emph{singular} part, defined as 
\begin{gather}
{\rm Reg}(u):=  \{ x\in \partial \{ u>0\} \cap \Omega: \mbox{ any blow up at $x$ is of the form } q_\nu\in \KK_+ \} \notag\\
{\rm Sing}(u):=\{ x\in \partial \{ u>0\} \cap \Omega: \mbox{ at least one blow up at $x$ is {\bf{not}} of the form } q_\nu\in \KK_+ \}\notag
\end{gather}
Their regularity is the content of the following results.

\begin{theo}[Uniqueness of the blow up and logarithmic convergence]\label{t:uniq}
	Let $ \gamma= \frac{d-1}{d+3}$, $\Omega\subset\R^d$ be an open set and $u\in H^1(\Omega)$ a minimizer of $\mathcal E$. Then the blow up of $u$ at each point of the free boundary $\partial \{ u>0\} \cap \Omega$ is unique. Moreover, the following convergence holds.
	\begin{enumerate}
		\item  For every $x_0\in \reg(u) \cap \Omega$ there exist $r:= r(x_0)$, $C:= C(x_0)$ and $\nu({x_0}) \in \R^d$, with $|\nu(x_0)|=\sfrac12$, such that
		$$\int_{\partial B_1} \left| u_{x_1,r} - q_{\nu(x_1)} \right| \, d\HH^{n-1} \leq C r^{\frac{(n+2) \eps}{2(1-\eps)}}\,, \quad \mbox{for every}\quad r \leq r_0 \mbox{ and } x_1\in \reg(u) \cap B_r(x_0).$$

		\item For every 
		$x_0 \in \sing(u)$,  there exist $r:= r(x_0)$, $C:= C(x_0)$ and $Q_{x_0} \in \KK$ such that 
		\begin{equation}
		\label{eqn:log-uniq}
		\int_{\partial B_1} \left| u_{x_0,r} - Q_{x_0}\right| \, d\HH^{n-1} \leq C (\log r )^{-\frac{1-\gamma}{2\gamma}} \,,\quad \mbox{for every}\quad r \leq r_0.
		\end{equation}
	\end{enumerate}
\end{theo}

The next regularity result recovers all the previously known results and improves the regularity of the singular set to $C^{1,\log}$. Before stating it we need to make precise what we mean by singular points of intermediate dimension. Given $k=0,\dots,d-1$, we define the \emph{singular set of dimension $k$} (also called $k$-th stratum) $S_k(u)$ as
\begin{align*}
S_k(u)
	&:=\{x\in \sing(u)\,:\,\dim(\ker(A))\leq k\,\mbox{ for every blow-up }\,Q_A\in \KK\,\mbox{ of $u$ at $x$}\}\\
	&=\bigcup_{l=1}^k\{x\in \sing(u)\,:\,\dim(\ker(A))= l\,\mbox{ for \emph{the unique} blow-up }\,Q_A\in \KK\,\mbox{ of $u$ at $x$}\}\,,
\end{align*}
where the equivalence of the two definitions is guaranteed by Theorem \ref{t:uniq}.
In the case of the stratum $S_0(u)$ the inequality \eqref{eqn:log-uniq} can be improved to $C^{1,\beta}$ convergence.

\begin{theo}[Regularity of the free boundary]\label{t:reg}
	Let $\eps>0$ be the constant from Theorem~\ref{t:epi:A}, $\beta =\frac{(d+2) \eps}{2(1-\eps)} \big( 1+ \frac{(d+2) \eps}{2(1-\eps)}\big)^{-1}$, $\Omega\subset\R^d$ be an open set and $u\in H^1(\Omega)$ a minimizer of $\mathcal E$. Then
	\begin{enumerate}
		\item $\reg(u)$ is locally the graph of a $C^{1,\beta}$ function; namely, for every $x_0\in \reg(u) \cap \Omega$ there exists $r:= r(x_0)$ such that $\reg(u) \cap B_{r}(x_0)$ is a $C^{1,\beta}$- submanifold of dimension $(d-1)$;
		
		\item For every $k=0,..., d-1$, $S_k(u)$ is contained in the union of countably many submanifolds of dimension $k$ and class $C^{1, log}$; namely for every $x_0\in \sing(u) \cap \Omega$ there exists $r_0:= r_0(x_0)$ and $C:= C(x_0)$ such that a logarithmic estimate holds
		\begin{equation}
		\label{eqn:log-cont-fb_intro}
		|Q(x_1) - Q(x_2)| \leq C (\log |x_1-x_2|)^{-\frac{1-\gamma}{2\gamma}} \qquad \mbox{for any } x_1,x_2 \in S_k \cap B_r(x_0).
		\end{equation}
		
		\item If the dimension $d=2$, then we have the estimate 
		\begin{equation}
		\label{eqn:log-cont-fb_intro_d2}
		|Q(x_1) - Q(x_2)| \leq C |x_1-x_2|^\beta \qquad \mbox{for any } x_1,x_2 \in S_k \cap B_r(x_0)\,,
		\end{equation}
		for $k=1,2$, where $\beta$ is the same as in (1). In particular $S_0$ consists of isolated points and $S_1$ is contained in the union of at most countably many curves of class $C^{1, \beta}$.
	\end{enumerate}
\end{theo}

\begin{oss}
	Thanks to a result of Caffarelli and Rivi\'ere (see \cite{CaRe}) it is possible to improve (3) to the following result: the boundary of a connected component
of the interior of the free-boundary is analytic except at finitely many singular points.
\end{oss}

  Theorems \ref{t:uniq} and \ref{t:reg} remain true if we consider a H\"older continuous weight function  $q:\Omega\to\R^+$ and more general functionals, for instance
  \begin{equation*}
  \mathcal{E}_q(u):=\int_\Omega \left[|\nabla u|^2+q(x) |u| \right]\,dx, \qquad \mathcal{A}_q(u):=\int_\Omega \left[\sqrt{|\nabla u|^2+1} +q(x) |u| \right]\,dx.
  \end{equation*}
  In this case, the regular and singular parts at a given point $x$ are defined as for $\mathcal E$, up to a constant which depends on $q(x)$. Given $u\in H^1(B_1)$ positive minimizer of $\mathcal E$, we define
  \begin{equation*}
  \begin{split}
  {\rm Reg}_q(u)&:=  \Big\{ x\in \partial \{ u>0\} \cap \Omega: \mbox{ any blow up at $x$ is of the form } q_\nu
  \mbox{ for } |\nu|=\frac{q(x)}2\Big\} ,
  \end{split}
  \end{equation*}
  \begin{equation*}
  \begin{split}
  {\rm Sing}_q(u)&:=  \Big\{ x\in \partial \{ u>0\} \cap \Omega: \mbox{ at least one blow up at $x$ is {\bf{not}} of the form } q_\nu,   \mbox{ for } |\nu|=\frac{q(x)}2 \Big\} ,
  \end{split}
  \end{equation*}
 $$ S_{q,k}(u)
	:=\{x\in \sing_q(u)\,:\,\dim(\ker(A))\leq k\,\mbox{ for every blow-up }\,Q_A\in \KK\,\mbox{ of $u$ at $x$}\}.
$$  
  \begin{theo}[H\"older continuous weight functions and area functional]\label{c:Q_funct}
  	Let $\alpha>0$, $\Omega\subset\R^d$ be an open set and  $q\in C^{0,\alpha}(\Omega;\R^+)$ be an H\"older continuous function such that $q\ge c_q>0$, where $c_q$ is a given constant. Let $u\in H^1(\Omega)$ be a minimizer of $\mathcal E_q$ or $\mathcal A_q$. Then the blow up of $u$ at each point of the free boundary $\partial \{ u>0\} \cap \Omega$ is unique and 
  	\begin{enumerate}
  		\item  there exists $\beta>0$ such that $\reg_q(u)$ is locally the graph of a $C^{1,\beta}$ function;
  		\item For every $k=0,..., d-1$, $S_{q,k}(u)$ is contained in the union of countably many submanifolds of dimension $k$ and class $C^{1, log}$; namely for every $x_0\in \sing_q(u) \cap \Omega$ there exists $r_0:= r_0(x_0)$ and $C:= C(x_0)$ such that a logarithmic estimate holds
  		\begin{equation}
  		\label{eqn:log-cont-fb}
  		|Q(x_1) - Q(x_2)| \leq C (\log |x_1-x_2|)^{-\frac{1-\gamma}{2\gamma}} \qquad \mbox{for any } x_1,x_2 \in \sing_q(u) \cap B_r(x_0).
  		\end{equation}
  	\end{enumerate}
  \end{theo}  
Compared to a similar result obtained from the epiperimetric inequality with indirect proof, here we have quantitative estimates as \eqref{eqn:log-cont-fb} and, for the regular set, an explicit H\"older regularity in terms of the dimension and the H\"older exponent of $q$.
  
\subsection*{Organization of the paper}
The paper is divided in four short sections. In Section \ref{s:preliminari} we fix notations and easy preliminary computations. In Section \ref{s:flat} we prove the Weiss epiperimetric inequality Theorem \ref{t:epi:A}, while Section \ref{s:sing} is dedicated to Theorem \ref{t:epi:B}. Finally, in Section \ref{s:regularity} we apply these two theorems to deduce the various regularity results.

\section{Preliminaries}\label{s:preliminari}
In this section we fix some notations and we recall some known facts about the solutions of the obstacle problem, their blow-up limits, the decomposition of the free boundary in a regular and singular part and its realtion with the Weiss boundary adjusted functional. The final subsection is dedicated to the Fourier analysis on the unit sphere in $\R^d$, which will be useful for both Theorems \ref{t:epi:B} and \ref{t:epi:A}.
\subsection{Notations}
We will use the following notations. $B_1$ is the $d$-dimensional unit ball centered in zero and $\omega_d=|B_1|$ is the Lebesgue measure of $B_1$. We denote by $\mathbb{S}^{d-1}$ or $\partial B_1$ the unit $(d-1)$-dimensional sphere in $\R^d$ equipped with the $(d-1)$-dimensional Hausdorff measure $\HH^{d-1}$. $\theta$ will denote the variable on the sphere $\partial B_1$. 
For an open set $\Omega$ in $\R^d$ or on the sphere $\mathbb{S}^{d-1}$ we will denote by $H^1(\Omega)\subset L^2(\Omega)$ the Sobolev space of weakly differentiable functions on $\Omega$ with gradients in $L^2(\Omega;\R^d)$ and by $H^1_0(\Omega)$ the space of functions $H^1(\Omega)$ which are zero on $\partial\Omega$. 

For a function $f:\R^d\to\R$ we denote by $f_+$ its positive part, $f_+(x)=\max\{f(x),0\}$. For instance, given a vector $\nu\in\R^d$ we will often use the notations 
$$(x\cdot\nu)_+=\max\{x\cdot\nu,0\}\qquad\text{and}\qquad (x\cdot\nu)_+^2=\big(\max\{x\cdot\nu,0\}\big)^2,$$ 
where $x\cdot\nu$ is the scalar product of the vectors $x$ and $\nu$ in $\R^d$.

\subsection{Weiss boundary adjusted energy}
For a function $u\in H^1(\Omega)$, with $\Omega\subset \R^d$, we denote by $W$, $W_0$ and $\W$ the functionals  
\begin{gather}
W_0(u,x_0,r):= \frac1{r^{d+2}} \int_{B_r(x_0)}|\nabla u|^2\,dx-\frac2{r^{d+3}}\int_{\partial B_r(x_0)}u^2\,d\HH^{d-1} \,,\notag\\
\W(u,x_0,r):=W_0(u,x_0,r)+\frac1{r^{d+2}}\int_{B_r(x_0)}u(x)\,dx,\notag\\
W(u,x_0,r):=W_0(u,x_0,r)+\frac1{r^{d+2}}\int_{B_r(x_0)}\max\{u(x),0\}\,dx,\notag
\end{gather}
where $x_0\in \de\{u>0\}$ and $0<r<\dist(x_0,\de \Omega)$, and we notice that for non-negative functions $u\in H^1(B_1)$ we have $\W(u,x_0,r)=W(u,x_0,r)$. In particular, we set 
$$W(u,0,r)=W(u,r)\qquad\text{and}\qquad W(u,1)=W(u),$$
and we recall the scaling property 
$$W(u,x_0,r)=W(u_{r,x_0})\,,\quad\text{where}\quad u_{r,x_0}(x)=\frac{u(rx+x_0)}{r^2}.$$ 

For any $u\in H^1(\Omega)$ the following identity holds for $x_0\in \de\{u>0\}$ and $0<r<\dist(x_0,\de \Omega)$
 \begin{equation}\label{e:Weiss_monotonicity}
\frac{d}{dr}W(u,x_0, r)=\frac{d+2}{r} \big[W(z_{r,x_0},1)-W(u_{r,x_0},1)\big]+\frac{1}{r}\int_{\partial B_1}|x\cdot \nabla u_{r,x_0}-2u_{r,x_0}|^2\,d\HH^{d-1}\,,
\end{equation}
where $ z_{r,x_0}(x):=|x|^2\,u_{r,x_0}\big(\frac{x}{|x|}\big)$ (see for instance \cite{Weiss2}).

\subsection{Global homogeneous solutions of the obstacle problem}\label{sub:Theta}
Wa say that the function $u_0:\R^d\to\R$ is a blow-up limit of $u$ in the point $x_0$, if 
$$u_0=\lim_{n\to\infty}u_{r_n,x_0}\qquad\text{for some sequence $(r_n)_n$ with} \, \lim_{n\to\infty}r_n=0,$$
where the converegnce is locally uniform and strong in $H^1_{loc}(\R^d)$.
Thanks to work of Caffarelli (see \cite{Caff_rev}), it is well known that $u_0$ is a global homogeneous solution of the obstacle problem. Precisely, $u_0\in \KK\cup\KK_+$ (introduced in \eqref{defn:K} and \eqref{defn:K+}). Moreover, we claim that 
\begin{equation}\label{e:WThetaK}
W(Q)=\frac{\omega_d}{8\, d (d+2)}=:\Theta\ \text{for every}\ Q\in\KK\,,\quad\text{and}\quad W(q)=\frac{\omega_d}{16\, d (d+2)}=:\Theta_+\ \text{for every}\ q\in\KK_+.
\end{equation}
Indeed, for every $Q_A\in \KK$ we have $\Delta Q_A=2\,\text{tr} A=\frac12$ and so an integration by parts gives
$$W_0(Q_A)=\int_{B_1}|\nabla Q_A|^2-2\int_{\partial B_1}Q_A^2=-\int_{B_1}Q_A\Delta Q_A=-\frac12\int_{B_1}Q_A.$$
Since $Q_A$ is positive and denoting $(a_{ij})_{ij}$ the coefficients of the matrix $A$, we get
$$W(Q_A)=\W(Q_A)=W_0(Q_A)+\int_{B_1}Q_A=\frac12\int_{B_1}Q_A=\frac12\int_{B_1}\sum_{i=1}^d a_{ii} x_i^2\,dx =\frac{\text{tr}A}2\int_{B_1}x_d^2\,dx =\Theta.$$
Analogously, for any $q_\nu\in \KK_+$ we have  $\Delta q_\nu=2\,|\nu|^2=\frac12$ on the set $\{x\cdot\nu>0\}$, so that
$$W(q_\nu)=\W(q_\nu)=W_0(q_\nu)+\int_{B_1}q_\nu=\frac12\int_{B_1}q_\nu
=\frac{|\nu|^2}2\int_{B_1\cap\{x_d>0\}}x_d^2\,dx =\frac{\omega_d}{16\, d (d+2)}=\Theta_+.$$

\subsection{Regular and singular free boundaries} We recall that, as observed by Weiss \cite{Weiss1}, a consequence of \eqref{e:Weiss_monotonicity} is that if $u\in H^1(\Omega)$ is a nonnegative minimizer of $\mathcal{E}$ in the open set $\Omega\subset\R^d$ and $x_0\in\Omega$, then the function $r \mapsto W(u,x_0,r)$ is nondecreasing (in its domain of definition $0<r<\text{dist}(x_0,\partial\Omega)$) and there exists the limit
\begin{equation}\label{e:Theta_u}
\Theta_u(x_0):=\lim_{r\to 0}W(u,x_0,r)=\inf_{r>0}W(u,x_0,r)=\lim_{r\to 0}W(u_{r,x_0})\,.
\end{equation}
Moreover, if $q$ is a blow-up limit of the minimizer $u$ in $x_0$, then 
$$W(q)=\lim_{n\to\infty}W(u_{r_n,x_0})=\Theta_u(x_0).$$
Since we have that $q\in\KK\cup\KK_+$, there are only two possible values for the energy density $\Theta_u(x_0)$: 
$$\Theta_u(x_0)=\Theta_+\qquad\text{or}\qquad \Theta_u(x_0)=\Theta.$$
Hence we can redefine the regular and the singular part of the free boundary $\partial\{u>0\}\cap\Omega$ as 
	\begin{equation*}
	\begin{split}
	{\rm Reg}(u)= \{ x\in \partial \{ u>0\} \cap \Omega: \Theta_u(x) = \Theta_+\},
	\end{split}
	\end{equation*}
	\begin{equation*}
	\begin{split}
	{\rm Sing}(u)=\{ x\in \partial \{ u>0\} \cap \Omega: \Theta_u(x) =\Theta\}.
	\end{split}
	\end{equation*}
By definition the free boundary $\partial \{ u>0\} \cap \Omega$ is a disjoint union of ${\rm Reg}(u)$ and ${\rm Sing}(u)$. Moreover, by the definition of the density \eqref{e:Theta_u} and the fact that $x_0\mapsto W(u,x_0,r)$ is continuous, the function $x_0\mapsto \Theta_u(x_0)$ is upper semicontinuous. This, together with \eqref{e:WThetaK} and the fact that all the blow-up limits are in $\KK\cup\KK_+$, finally gives the following characterization of ${\rm Reg}(u)$ and ${\rm Sing}(u)$:
\begin{itemize}
\item the set ${\rm Reg}(u)$ is a relatively open subset of the free boundary $\partial \{ u>0\}$, and every blow-up limit at a point of ${\rm Reg}(u)$ is of the form $q_\nu$, for some $q_\nu\in\KK_+$;
\item  the set ${\rm Sing}(u)$ is closed, and every blow-up limit at a point of ${\rm Sing}(u)$ is of the form $Q_A$, for some $Q_A\in\KK$.
\end{itemize}

\subsection{Eigenvalues and eigenfunctions on subdomains of the sphere}\label{sub:spectrum} Let $S\subseteq \mathbb{S}^{d-1}$ be an open set. Let $0<\lambda_1\le \lambda_2\le \dots\le\lambda_j\le \dots$ be the eigenvalues (counted with multiplicity) of the spherical Laplace-Beltrami operator with Dirichlet conditions on $\partial S$ and $\{\phi_j\}_{j\ge1}$ be the corresponding eigenfunctions, that is the solutions of the problem
\begin{equation}
\label{defn:eigenval-in-s}
-\Delta_{\mathbb S^{d-1}} \phi_j=\lambda_j\phi_j\quad\text{in}\quad S,\qquad \phi_j=0\quad\text{on}\quad \partial S,\qquad \int_S\phi_j^2(\theta)\,d\HH^{d-1}(\theta)=1.
\end{equation}
Any function $\psi \in H_0^1(S)$ can be decomposed as $ \psi(\theta)=\sum_{j=1}^{\infty}c_j\phi_j(\theta)$. The following lemma compares the energies of $2$-homogeneous and $\alpha$-homogeneous functions by means of the Fourier decomposition of their common values on $\partial B_1$.

\begin{lemma}\label{l:fourier}
	Let $\psi\in H_0^1(S)$ and consider the $2$-homogeneous extension $\varphi (r,\theta)=r^2\psi(\theta)$ and the $\alpha$-homogeneous extension  $\tilde \varphi (r,\theta)=r^\alpha\psi(\theta)$ respectively of $\psi$ to $B_1$, for some $\alpha>2$. Set
	\begin{equation}\label{eps_choice}
	\eps_\alpha:=\frac{\alpha-2}{d+\alpha}\qquad\text{and}\qquad \lambda_\alpha:=\alpha(\alpha+d-2)\,.
	\end{equation}
	Then the following inequality holds 
	\begin{equation}\label{e:W_0}
	W_0(\tilde\varphi)-(1-\eps_\alpha)W_0(\varphi)= \frac{\eps_\alpha}{d+2\alpha-2}\sum_{j=1}^\infty (-\lambda_j+\lambda_\alpha) c_j^2\,.
	\end{equation} 
\end{lemma}

\begin{proof}
Since $\| \varphi_j\|_{L^2(\partial B_1)} =1$ and $\|\nabla_{\theta} \varphi_j\|_{L^2(\partial B_1)} =\lambda_j$ for every $j \in \{0\} \cup \N$, the energy of the $\alpha$-homogeneous function $\tilde\varphi(r,\theta)=r^\alpha \psi(\theta)$ can be written as
	\begin{align*}
	W_0(\tilde \varphi)&= \sum_{j=1}^\infty c_j^2\left(\int_0^1 r^{d-1}\,dr\int_S\,d\HH^{d-1} \left[\alpha^2r^{2\alpha-2}\phi_j^2+r^{2\alpha-2}|\nabla_\theta\phi_j|^2\right]-2\int_S \phi_j^2\,d\HH^{d-1}\right)\\
	&= \sum_{j=1}^\infty c_j^2\left(\frac{\alpha^2+\lambda_j}{d+2\alpha-2}-2\right).
	\end{align*} 
When $\alpha=2$ and $\varphi(r,\theta)=r^2 \psi(\theta)$, we get 
	$$W_0(\varphi)=\sum_{j=1}^\infty c_j^2\left(\frac{4+\lambda_j}{d+2}-2\right).$$
	We now notice that for every $\lambda$ we have
	$$\Big(\frac{\alpha^2+\lambda}{d+2\alpha-2}-2\Big)-(1-\eps_\alpha)\Big(\frac{4+\lambda}{d+2}-2\Big)=\frac{\lambda(2-\alpha)}{(d+\alpha)(d+2\alpha-2)}+\frac{(\alpha-2)\alpha (\alpha+d-2)}{(d+\alpha)(d+2\alpha-2)},$$
	which concludes the proof of Lemma \ref{l:fourier}. 
\end{proof}
The above lemma, in particular, shows that if the decomosition of $\psi$ involves only eigenfunctions corresponding to eigenvalues $\lambda_j\ge\lambda_\alpha$, then the $\alpha$-homogeneous extension $\tilde\varphi$ has a strictly lower energy than the two-homogeneous extension $\varphi$. In order to choose appropriately $\alpha$ we will need some additional information on the spectrum of the Laplacian on $S$. We recall that the function $\phi_j:S\to\R$ is a solution of the first equation in \eqref{defn:eigenval-in-s} if and only if its $\alpha_j$-homogeneous extension $\varphi_j(r,\theta)=r^{\alpha_j}\phi_j(\theta)$ is harmonic in the cone 
$\{(r,\theta)\in \R^+\times\partial B_1\ :\ \theta\in S\},$
where the homogeneity $\alpha_j$ is uniquely determined by the identity 
$\lambda_j=\alpha_j(\alpha_j+d-2).$

{\bf The spectrum on the sphere $\mathbb{S}^{d-1}$.} By the fact that the  homogeneous harmonic functions in $\R^d$ are necessarily polynomials, we have that: 
\begin{itemize}
\item $\lambda_1=0$ and the corresponding eigenfunction is the constant $\phi_1=|\partial B_1|^{-1/2}= (d\omega_d)^{-1/2}$. 
\item $\lambda_2=\dots=\lambda_{d+1}=d-1$, the corresponding homogeneity constants are $\alpha_2=\dots=\alpha_{d+1}=1$ and the corresponding
eigenspace  coincides with the space of linear functions in $\R^d$. 
\item $\lambda_{d+2}=\dots=\lambda_{d(d+3)/2}=2d$, the corresponding homogeneity constants are $\alpha_{d+2}=\dots=\alpha_{d(d+3)/2}=2$. The corresponding eigenspace has dimension $d(d-1)/2$ and is generated by the (restrictions to $\mathbb{S}^{d-1}$ of the) two-homogeneous harmonic polynomials:
$$E_{2d}=\{Q_A \colon \R^d \to \R \,:\, Q_A(x) = x \cdot Ax,\,\text{ $A$  symmetric with }{\rm{tr}} A = 0\}.$$
\item If $j> d(d+3)/2$ (that is $\lambda_j>2d$), then $\lambda_j\ge 3(3+d-2)=3(d+1)$. 
\end{itemize}  
{\bf The spectrum on the half-sphere $\partial B_1^+=\{x_d>0\}\cap \partial B_1$.} We notice that the odd extension (with respect to the plane $\{x_d=0\}$) of any eigenfunction $\phi_j$ on the half-sphere $\partial B_1^+$ is an eigenfunction on the entire sphere $\partial B_1$, which is zero on the equator $\{x_d=0\}\cap \partial B_1$. Thus, one can easily deduce that:
\begin{itemize}
\item $\lambda_1=d-1$ and the corresponding eigenfunction is $\phi_1(x)=\frac{x_d}{\sqrt\omega_d}.$
\item $\lambda_2=\dots=\lambda_{d}=2d$, the corresponding homogeneity constants are $\alpha_2=\dots=\alpha_{d}=2$ and the corresponding eigenspace $E_{2d}$ has dimension ($d-1$) and is generated by the polynomials 
$$Q_j(x)=x_d x_{j-1},\qquad\text{for every}\qquad j=2,\dots,d.$$
\item If $j> d$ (that is $\lambda_j>2d$), then $\lambda_j\ge 3(3+d-2)=3(d+1)$. 
\end{itemize}  
{\bf The spectrum on the spherical cap $S_{\delta}=\partial B_1\cap\{x_d>-\delta\}$.} We first notice that the spectrum $\{\lambda_j(\delta)\}_{j\ge 1}$ of the spherical cap $S_\delta$ varies continuously with respect to $\delta$. Thus, for $\delta>0$ small enough (smaller than some dimensional constant), we have 
\begin{itemize}
\item $\lambda_1(\delta)$ is simple (isolated) eigenvalue and $\lambda_1(\delta)\le d-1$;
\item $d-1< \lambda_j(\delta)<2d$, for every $j=2,\dots,d$;
\item $\lambda_j(\delta)\ge 3d$, for every $j> d$. 
\end{itemize}
Moreover, a standard separation of variables argument gives that:
\begin{itemize}
\item the first eigenfunction $\phi_1$ on $S_{\delta}$ is positive and depends only on the first variable $x_d$, that is $\phi_1(x)=\phi_1(x_d).$ 
\item the eigenfunctions $\phi_2,\dots,\phi_d$ correspond to the same eigenvalue $\lambda_2(\delta)=\dots=\lambda_d(\delta)$ and there is  a function $\phi=\phi(x_d)$ such that 
$$\phi_j(x)=x_{j-1} \phi(x_d)\quad \text{for every}\quad j=2,\dots,d.$$ 
\end{itemize}

\section{The epiperimetric inequality at flat points: proof of Theorem \ref{t:epi:A}}\label{s:flat}

In order to prove Theorem \ref{t:epi:A} we decompose the function $z$ as
$$z=q_\nu+\varphi,$$
where $q_\nu(x)=(x\cdot\nu)_+^2$ and $\nu\in\R^d$ to be chosen later.
We then replace the $2$-homogeneous function $\varphi (r,\theta)=r^2\phi(\theta)$ by an $\alpha$-homogeneous function $\tilde\varphi(r,\theta)=r^\alpha\phi(\theta)$, for some $\alpha>2$. The final competitor will be of the form 
$$h=q_\nu+\tilde\varphi,$$ 
and $\eps=\eps_\alpha$ will be given by 
\begin{equation}\label{e:eps}
\eps_\alpha:=\frac{\alpha-2}{d+\alpha}.
\end{equation}
We notice that the competitor $h$ is non-negative in $B_1$, thus we only need to prove the inequality 
\begin{equation}\label{e:tildepi}
 \W(h)-\Theta_+-(1-\eps)\left(\W(z)-\Theta_+\right) \le 0.
\end{equation}
We divide the proof into three steps. 

\smallskip

\noindent {\emph{Step 1}.} Using the properties of $q_\nu$, we first reduce the inequality \eqref{e:tildepi} to a comparison of the energy of  $\tilde\varphi$ to the one of $\varphi$. Precisely, in Subsection \ref{sub:decomposition}, we prove the inequality 
\begin{equation}\label{e:mainest0}
\W(h)-\Theta_+-(1-\eps)\Big(\W(z)-\Theta_+\Big) \le W_0(\tilde \varphi)-(1-\eps)W_0(\varphi).
\end{equation}

\smallskip 

\noindent {\emph{Step 2}.}   In Subsection \ref{sub:choice} we prove that we can choose $q_\nu$ in such a way that the function $\phi:=c-q_\nu$ does not contain modes of the first $d$ eigenvalues on the spherical cap $S_{\delta_0}$. Precisely, we prove the following claim. For every  $\delta_0>0$ there exists  $\delta>0$ such that
\begin{equation}\label{e:claim}
\begin{array}{ll}
\ds\quad\text{for every}\ c\in H^1_0(S_{\delta_0})\ \text{satisfying}\ \|c-q_{e_d/2}\|_{L^2(\partial B_1)} \le \delta\ \text{there exists}\ \nu\in\R^d\ \text{such that}\\
\ds\quad u_\lambda\in H^1_0(S_{\delta_0})\ \text{and}\ \int_{S_{\delta_0}}c(\theta)\phi_j(\theta)\,d\HH^{d-1}(\theta)=\int_{S_{\delta_0}}q_\nu(\theta)\phi_j(\theta)\,d\HH^{d-1}(\theta),\ \text{for every}\ j=1,\dots,d,   
\end{array}
\end{equation}
where $\phi_1,\dots,\phi_d$ are the first $d$, orthonormal in $L^2(\partial B_1)$, eigenfunctions of the Laplace-Beltrami operator on $S_{\delta_0}$ with Dirichlet boundary conditions on $\partial S_{\delta_0}$.  

\smallskip

\noindent {\emph{Step 3}.} In Subsection \ref{sub:fourier} we use Lemma \ref{l:fourier} and the choice of $\nu$ from Step 2 to prove the inequality 
\begin{equation}\label{e:mainest2}
\ds W_0(\tilde \varphi)-(1-\eps)W_0(\varphi)\le 0,
\end{equation}
which together with \eqref{e:mainest0} gives \eqref{e:tildepi}.

\subsection{Decomposition of the energy}\label{sub:decomposition}
We prove \eqref{e:mainest0} in the following lemma:
\begin{lemma}\label{l:decomposition}
Let $\alpha>2$,  $\eps_\alpha$ as in \eqref{e:eps}, $\nu=(\nu_1,\dots,\nu_d)\in\R^d$ and $q_\nu(x)=(x\cdot\nu)_+^2$. Suppose that $\phi\in H^1(\partial B_1)$, $\varphi (r,\theta)=r^2\phi(\theta)$ and $\tilde\varphi(r,\theta)=r^\alpha\phi(\theta)$. Then 
\begin{equation}\label{e:mainest}
\ds \Big(\W(q_\nu+\tilde\varphi)-\Theta_+\Big)-(1-\eps_\alpha)\Big(\W(q_\nu+\varphi)-\Theta_+\Big)\le W_0(\tilde \varphi)-(1-\eps_\alpha)W_0(\varphi).
\end{equation}
\end{lemma}
\begin{proof}
Suppose, without loss of generality that, $\ds q_\nu(x)=c_0q(x)$, where for the sake of simplicity we set $q:=q_{e_d/2}\in\KK_+$. 
Notice that for every $\psi\in H^1(B_1)$ we have 
\begin{align*}
\W(c_0q+\psi)-&\Theta_+=c_0^2\int_{B_1}|\nabla q|^2-2c_0^2\int_{\partial B_1}q^2+c_0\int_{B_1}q-\frac12\int_{B_1}q\\
&\qquad\qquad+2c_0\left(\int_{B_1}\nabla q\cdot\nabla\psi-2\int_{\partial B_1}q\psi\right)
+\int_{B_1}|\nabla \psi|^2-2\int_{\partial B_1}\psi^2
+\int_{B_1}\psi\\
&=-\frac{(c_0-1)^2}2\int_{B_1}q+2c_0\left(-\int_{B_1}\Delta q\,\psi+\int_{\partial B_1} \partial_r q\,\psi-2\int_{\partial B_1}q\psi\right)+W_0(\psi)+\int_{B_1}\psi\\
&=-(c_0-1)^2\Theta_++W_0(\psi)+\int_{B_1}\psi-c_0\int_{B_1^+}\psi\,,
\end{align*} 
where we used that
$
\Theta_+=\frac12 \int_{B_1}q$,  $\Delta q=\frac12\,\chi_{B_1^+}$
and $\de_r q=2\,q$.
If $\psi=\tilde\varphi=r^2\,\phi$, then we have 
$$\int_{B_1}\tilde\varphi-c_0\int_{B_1^+}\tilde\varphi=\frac1{d+\alpha}\left(\int_{\partial B_1}\phi-c_0\int_{\partial B_1^+}\phi\right)=:\frac1{d+\alpha}\beta(\phi),$$
and we can write the energy of $c_0q+\tilde\varphi$ in the form
\begin{align*}
\W(c_0q+\tilde\varphi)-\Theta_+=-(c_0-1)^2\Theta_++W_0(\tilde\varphi)+\frac{1}{d+\alpha}\beta(\phi).
\end{align*} 
Applying the above estimate to $\varphi$ and $\tilde\varphi$ and thanks to the definition of $\eps_\alpha$, we get  
\begin{align*}
\W(c_0q+\tilde \varphi)&-\Theta_+-(1-\eps_\alpha)\big(\W(c_0q+ \varphi)-\Theta_+\big)\\
&=-\eps_\alpha (c_0-1)^2\Theta_++W_0(\tilde \varphi)-(1-\eps_\alpha)W_0(\varphi)+\left(\frac1{d+\alpha}-\frac{1-\eps_\alpha}{d+2}\right)\beta(\phi)\\
&=-\eps_\alpha (c_0-1)^2\Theta_++W_0(\tilde \varphi)-(1-\eps_\alpha)W_0(\varphi),
\end{align*} 
which concludes the proof of Lemma \ref{l:decomposition}. 
\end{proof}

\subsection{Choice of $\nu$}\label{sub:choice}
In this section we prove the claim \eqref{e:claim}, which is a straightforward consequence of the following lemma.
\begin{lemma}\label{l:inversa}
Given $\delta_0>0$, we denote by $S_{\delta_0}$ the set $\{x_d>-\delta_0\}\cap\partial B_1$ and by $\phi_1,\dots,\phi_d$ the first $d$ eigenfunctions on the set $S_{\delta_0}$. Then the function 
$$F:\R^d\to\R^d,\qquad F(\nu) = \left(\int_{S_{\delta_0}}q_\nu\phi_1,\dots,\int_{S_{\delta_0}}q_\nu\phi_d \right),$$
is a $C^1$ diffeomorphism in a neighbourhood $U\subset\R^d$ of $\ds\frac{e_d}2$. 
\end{lemma}
\begin{proof}[Proof of Lemma \ref{l:inversa}]

We first notice that $F$ is a $C^1$ function in a neighborhood of $e_d/2$, because the function 
$\R^d\times\R^d\ni (x,y)\mapsto (x\cdot y)_+^2$
is $C^1$. We now calculate the partial derivatives of $F=(F_1,\dots,F_d)$ in $e_d/2$. Using the fact that the first eigenfunction depends only on one varibale, $\phi_1=\phi_1(x_d)$, and that the higher eigenfunctions can be written in the form $\phi_j(x)=x_{j-1}\phi(x_d)$, for every $j=2,\dots,d$ (see Subsection \ref{sub:spectrum}), we get that 
$$F_1(\nu)=\int_{\partial B_1}q_\nu(x)\,\phi_1(x_d)\,dx\qquad\text{and}\qquad F_j(\nu)=\int_{\partial B_1}q_\nu(x)\,x_{j-1}\phi(x_d)\,dx,\quad\forall j=2,\dots,d. $$
Setting $\partial B_1^+=\{x_d>0\}\cap\partial B_1$ we have 
$$\frac{\partial F_1}{\partial \nu_d}(e_d/2)=\int_{\partial B_1^+}x_d^2\phi_1(x_d)\,dx>0,$$
$$\frac{\partial F_1}{\partial \nu_j}(e_d/2)=\int_{\partial B_1^+}x_d x_j\phi_1(x_d)\,dx=0,\  \forall j=1,\dots,d-1.$$
where the positivity of the first term follows from the positivity of $\phi_1$, while the second term is zero since $x_j$ is odd. Moreover, for every $j=2,\dots,d$ and $i=1,\dots,d$, we have 
\begin{align*}
\frac{\partial F_j}{\partial \nu_i}(e_d/2)&=\int_{\partial B_1^+}\!\!\!x_i x_d x_{j-1}\phi(x_d)\,dx=\delta_{i(j-1)}\int_{\partial B_1^+} \!\!\!x_i^2 x_d \phi(x_d)\,dx=\frac{\delta_{i(j-1)}}{d-1}\int_{\partial B_1^+} \!\!\!(1-x_d^2) x_d\phi(x_d)\,dx\,,
\end{align*}
where we used the fact that $x_i$ and $x_j$ are odd for the first equality, and $\int x_j^2=\frac{1}{d-1}\sum_{j=1}^{d-1}\int x_j^2=\frac1{d-1}\int(1-x_d^2)$.
By the positivity of $\phi$ and the fact that $\lim_{\delta_0\to0}\|\phi-c_2x_d^+\|_{L^2(\partial B_1)}=0$, for dimensional constants $c_1$ and $c_2$ (which is due to the fact that on the half-sphere $\partial B_1^+$ the eigenfunctions are of the form $\phi_j(x)=c_2x_dx_{j-1}$ for $j=2,\dots,d$) we get that for $\delta_0$ small enough $DF(e_d/2)$ is an invertible matrix and so, by the inverse function theorem there is a neighborhood of $e_d/2$ on which $F$ is a $C^1$ diffeomorphism. 
\end{proof}

\subsection{Homogeneity improvement of $\varphi$}\label{sub:fourier}

We prove \eqref{e:mainest2}. Indeed, by the fact that the Fourier expansion of $\phi(\theta):=c(\theta)-q_\nu(\theta)$ does not contain the first $d$ modes $\phi_1,\dots,\phi_d$ on the spherical cap $S_{\delta_0}$ (claim \eqref{e:claim}), we obtain that the function $\phi$ can be expanded in Fourier series as
$$\phi(\theta)=\sum_{j=d+1}^{\infty}c_j\phi_j(\theta)\quad\text{on the spherical cap}\quad S_{\delta_0}=\partial B_1\cap \{x_d>-\delta_0\}.$$
Thus, by Lemma \ref{l:fourier} we get
\begin{align*}
W_0(\tilde\varphi)-(1-\eps_\alpha)W_0(\varphi)&=\frac{\eps_\alpha}{d+2\alpha-2}\sum_{j=d+1}^\infty (-\lambda_j+\lambda_\alpha) c_j^2,
\end{align*}
where $\lambda_j$ are the eigenvalue of the Dirichlet Laplacian on $S_{\delta_0}$ and $\lambda_\alpha=\alpha(\alpha+d-2)$. 
On the other hand, for $\delta_0>0$ small enough, we have that $\lambda_j\ge 3d$,  for $j\ge d+1$ (see Subsection \ref{sub:spectrum}), so that $-\lambda_j+\lambda_\alpha\leq 0$ whenever $\alpha>2$ and $\alpha (\alpha+d-2)\le 3d$.
. Thus, choosing for instance
$$\alpha=\frac52\qquad\text{and}\qquad \eps=\frac{\alpha-2}{d+\alpha}=\frac{1}{2d+5},$$ 
we conclude the proof of \eqref{e:mainest2} and Theorem \ref{t:epi:A}.

\section{The epiperimetric inequality for singular points: proof of Theorem \ref{t:epi:B}}\label{s:sing}

We first notice that given any two-homogeneous function $z(r,\theta)=r^2\,c(\theta)$, we can decompose it in Fourier series on the sphere $\partial B_1$ as 
\begin{align*}
c(\theta)=\sum_{j=1}^\infty c_j \phi_j(\theta)=\; c_1\phi_1(\theta)\quad+\sum_{\{j\,:\,\lambda_j=d-1\}} c_j\,\phi_j(\theta)\quad+\sum_{\{j\,:\,\lambda_j=2d\}} c_j\,\phi_j(\theta)\quad+\sum_{\{j\,:\, \lambda_j>2d\}} c_j \phi_j(\theta)\,.
\end{align*}
Therefore $z$ can be decomposed in a unique way as
$$z=q_\nu+Q_A+\varphi,$$
where 
\begin{enumerate}[(i)]
\item $\nu\in\R^d$ is such that $q_\nu(x)=(x\cdot\nu)_+^2$ contains in its Forurier expansion precisely the sum $$\ds\sum_{\{j\,:\,\lambda_j=d-1\}} c_j\,\phi_j(\theta) \ ;$$
\item $A$ is a symmetric matrix depending on the coefficients $c_j$, corresponding to the eigenvalues $\lambda_j=0,d-1,2d$, and $Q_A(x)=x\cdot Ax$;
\item $\varphi$ is a two-homogeneous function, in polar coordinates $\varphi(r,\theta)=r^2\phi(\theta)$, containing only higher modes on $\partial B_1$, that is the trace $\phi$ can be written in the form  
\begin{equation}\label{e:fourier_phi}
\phi(\theta)=\sum_{\{j\,:\, \lambda_j>2d\}} c_j \phi_j(\theta),
\end{equation}
where $\{\phi_j\}_{j\in\N}$ are the eigenfunctions of the spherical laplacian as in \eqref{defn:eigenval-in-s} with $S= \partial B_1$.
\end{enumerate}

Notice that, in the above representation $A$ might not be positive definite. Let $B$ be a symmetric positive definite matrix and $Q_B(x)=x\cdot Bx$. Then, $z$ can be rewritten as 
$$z=q_\nu+ Q_B+ (Q_A-Q_B)+\varphi\,.$$
We then replace the $2$-homogeneous function $\psi:=(Q_A-Q_B)+\varphi$ by an $\alpha$-homogeneous function $\tilde\psi$ with the same boundary values as $\psi$. 
We will choose $\alpha>2$ such that
\begin{equation}\label{e:choice_of_alpha}
\eps_\alpha:=\frac{\alpha-2}{d+\alpha}=\ds\eps \left(C_4\|\nabla_\theta \phi\|_{L^2(\partial B_1)}^{2}\right)^\gamma,
\end{equation}
where $C_4$ is the dimensional constant from the inequality \eqref{e:sing:final} and $\gamma$ is the constant from \eqref{e:epi_flat_point_sing}. Subsequently we will choose $\eps$ to be small enough, but yet depending only on the dimension.  

Finally, the competitor $h$ is given by
$$h=q_\nu+Q_B+ \tilde\psi.$$ 
Since $\inf\{\psi,0\}\le \inf\{\tilde\psi,0\}$ and $q_\nu+Q_B\ge 0$, by the choice of $B$, $h$ is non-negative in $B_1$ and so we only need to prove the inequality
$$
W(h)\le W(z)-\eps\big(W(z)-\Theta\big)^{1+\gamma}\,.
$$ 
The proof of Theorem \ref{t:epi:B} will be carried out in four steps.\\
\smallskip 

\noindent {\emph{Step 1.}} In Subsection \ref{sub:sing:decomposition} we set  $ c_0=4\sum_{j=1}^d\nu_j\ $ and $\ \ds b=4\,\text{tr}B$ and we prove the identity 
\begin{align}
W(h)&-\Theta-(1-\eps_\alpha)\left(W(z)-\Theta\right)\label{e:sing:energy}\\
&=-\frac{\eps_\alpha}2\left((1-b-c_0)^2+(1-b)^2\right)\Theta+W_0(\tilde\psi)-(1-\eps_\alpha)W_0(\psi).\nonumber
\end{align}

\smallskip

\noindent{\emph{Step 2.}} We now choose $Q_B$. If $A$ is positive definite ($Q_A\ge 0$), then we choose $B=A$. If $Q_A$ changes sign, then up to a change of coordinates we can assume that there exist $ a_j\ge 0$ for every $j=1,\dots,d $ such that
$$Q_A(x)=-\sum_{j=1}^k a_jx_j^2+\sum_{j=k+1}^da_j x_j^2, \qquad \quad 
 a_d\ge \frac1{4d}> \sum_{j=1}^k a_j,$$
where the last inequality being due to the fact that $Q_A$ is $L^2(\partial B_1)$-close to the set of admissible blow-ups $\KK$. 
We set
\begin{equation}\label{e:Q_B}
Q_B(x):=\sum_{j=k+1}^d a_jx_j^2-\Big(\sum_{j=1}^k a_j\Big)x_d^2\geq 0\,
\end{equation}
where the last inequality, that is the positive definiteness of $B$, depends on $a_d\ge \frac1{4d}> \sum_{j=1}^k a_j$.
In Subsection \ref{sub:sing:main} we then prove that there exists a dimensional constant $C_2>0$ such that
\begin{equation}\label{e:sing:main}
W_0(\tilde\psi)-(1-\eps_\alpha)W_0(\psi)\le \eps_\alpha^2 C_2\sum_{j=1}^k a_j^2-\frac{\eps_\alpha}{2} \|\nabla_\theta \phi\|_{L^2(\partial B_1)}^2.
\end{equation}

\noindent{\it Step 3.} In Subsection \ref{sub:sing:pre-final} we prove that there are dimensional constants $C_3>0$ and $\gamma\in[0,1)$ such that 
\begin{equation}\label{e:sing:pre-final}
\sum_{j=1}^k a_j^2\le C_3\|\nabla_\theta \phi\|_{L^2(\partial B_1)}^{2(1-\gamma)}.
\end{equation}
In the supplementary Subsection \ref{sub:sing:complementary} we show that this estimate can be improved in several ways:
\begin{enumerate}[-]
\item in dimension two \eqref{e:sing:pre-final} holds with $\gamma=0$;
\item the dimensional constants $\gamma$ can be replaced by a (smaller) constant $\gamma_k$, this time depending on $d$ and $k$. In the two extremal cases $k=0$ and $k=d-1$ the constant is zero. \\
\end{enumerate}

\noindent{\it Step 4.} In Subsection \ref{sub:sing:final} we prove that there is a dimensional constant $C_4$, such that 
\begin{equation}\label{e:sing:final}
W(z)-\Theta\le C_4\|\nabla_\theta \phi\|_{L^2(\partial B_1)}^2.
\end{equation}

\noindent{\it Conclusion of the proof.} The proof of Theorem \ref{t:epi:B} now follows directly by  \eqref{e:sing:energy}, \eqref{e:sing:main}, \eqref{e:sing:pre-final}  and \eqref{e:sing:final}. Indeed, by \eqref{e:sing:main} and \eqref{e:sing:pre-final} we get that 
\begin{align*}
W_0(\tilde\psi)-(1-\eps_\alpha)W_0(\psi)&\le \eps_\alpha^2 C_2\sum_{j=1}^k a_j^2-\frac{\eps_\alpha}2\|\nabla_\theta \phi\|_{L^2(\partial B_1)}^2\\
&\le \eps_\alpha^2 C_2C_3\|\nabla_\theta \phi\|_{L^2(\partial B_1)}^{2(1-\gamma)}-\frac{\eps_\alpha}2\|\nabla_\theta \phi\|_{L^2(\partial B_1)}^2.
\end{align*}
By the definition of $\eps_\alpha$ and \eqref{e:sing:final} we get
\begin{align*}
W_0(\tilde\psi)-(1-\eps_\alpha)W_0(\psi)&\le \eps^2C_4^{2\gamma}\|\nabla_\theta \phi\|_{L^2(\partial B_1)}^{4\gamma} C_2C_3\|\nabla_\theta \phi\|_{L^2(\partial B_1)}^{2(1-\gamma)}-\frac{\eps}2 C_4^{\gamma}\|\nabla_\theta \phi\|_{L^2(\partial B_1)}^{2\gamma} \|\nabla_\theta \phi\|_{L^2(\partial B_1)}^2\\
&= \eps C_4^{\gamma}\left(\eps C_2C_3 C_4^{\gamma}-\frac12\right) \|\nabla_\theta \phi\|_{L^2(\partial B_1)}^{2+2\gamma},
\end{align*}
which is negative for $\eps$ small enough (but yet, $\eps$ depends only on the dimension). Finally, by \eqref{e:sing:energy}, the definition of $\eps_\alpha$ and \eqref{e:sing:final} we obtain
\begin{align*}
W(h)-\Theta&\le \left(1-\eps_\alpha\right)\big(W(z)-\Theta\big)
= \left(1-\eps\, C_4^\gamma\, \|\nabla_\theta \phi\|_{L^2(\partial B_1)}^{2\gamma}\right)\big(W(z)-\Theta\big)\\
& \le \Big(1-\eps\,\big(W(z)-\Theta\big)^\gamma\Big)\big(W(z)-\Theta\big),
\end{align*}
which is precisely \eqref{e:epi_flat_point_sing}. We now proceed with the proof of \eqref{e:sing:energy}, \eqref{e:sing:main}, \eqref{e:sing:pre-final}  and \eqref{e:sing:final}.

\subsection{Decomposition of the energy}\label{sub:sing:decomposition} We prove the following lemma, which implies easily \eqref{e:sing:energy}.
\begin{lemma}\label{l:sing:decomposition}
Let $\alpha>2$ and $\eps_\alpha=\frac{\alpha-2}{d+\alpha}$; let $0\neq \nu=(\nu_1,\dots,\nu_d)\in\R^d$, $q_\nu(x)=(x\cdot\nu)_+^2$, $c_0=4\sum_{j=1}^d\nu_j^2$; let $B$ be a symmetric matrix with and $b=4\,\text{tr}B\neq0$ and $Q_B(x)=x\cdot Bx$. Suppose that $\phi\in H^1(\partial B_1)$, $\varphi (r,\theta)=r^2\phi(\theta)$ and $\tilde\varphi(r,\theta)=r^\alpha\phi(\theta)$. Then 
\begin{align*}
\ds \W(q_\nu+Q_B+\tilde\varphi)&-\Theta-(1-\eps_\alpha)\Big(\W(q_\nu+Q_B+\varphi)-\Theta\Big)\\
&= -\frac{\eps_\alpha}2\left((1-b-c_0)^2+(1-b)^2\right)\Theta+W_0(\tilde \varphi)-(1-\eps_\alpha)W_0(\varphi).
\end{align*}
\end{lemma}
\begin{proof}
We $Q=\frac1b Q_B$. Thus, we have $Q\in\KK$. In particular, $\Delta Q=1/2$ in $B_1$ and $W(Q)=\Theta$. 
We notice that for every function $\eta\in H^1(B_1)$, we have 
\begin{align*}
\W(bQ+\eta)-\Theta&=b^2W_0(Q)+W_0(\eta)+2b\left(\int_{B_1}\nabla Q\cdot\nabla\eta-2\int_{\partial B_1}Q\eta\right)+b\int_{B_1}Q+\int_{B_1}\eta-\Theta\\
&=b^2W_0(Q)+W_0(\eta)-2b\int_{B_1}\eta\Delta Q+b\int_{B_1}Q+\int_{B_1}\eta-\Theta,
\end{align*}
which gives 
\begin{equation}\label{e:saluta_andonio}
\W(bQ+\eta)-\Theta=-(1-b)^2\Theta+W_0(\eta)+(1-b)\int_{B_1}\eta.
\end{equation}
We set $q=\frac1{c_0}q_\nu$. Thus $q\in\KK_+$ and $W(q)=\Theta_+=\frac12\int_{B_1} q$. 
Setting $\eta=c_0q+\psi$ in \eqref{e:saluta_andonio} we obtain 
\begin{align}
\W(bQ+c_0q+\psi)&-\Theta=-(1-b)^2\Theta+W_0(c_0q+\psi)+(1-b)\int_{B_1}(c_0q+\psi)\nonumber\\
&=-(1-b)^2\Theta+c_0^2W_0(q)+W_0(\psi)\nonumber\\
&\quad +2c_0\left(\int_{B_1}\nabla q\cdot\nabla\psi-2\int_{\partial B_1}q\psi\right)+(1-b)\int_{B_1}(c_0q+\psi)\nonumber\\
&=-(1-b)^2\Theta-c_0^2\frac{\Theta}2+W_0(\psi) -2c_0\int_{B_1}\psi\Delta q+(1-b)c_0\Theta+(1-b)\int_{B_1}\psi\nonumber\\ 
&=-\frac{(1-b-c_0)^2+(1-b)^2}2\Theta+W_0(\psi)+\beta(\psi),
\label{e:sing:decomposition}
\end{align}
where in the last line we set 
$$\beta(\psi):=(1-b)\int_{B_1}\psi-c_0\int_{B_1^+}\psi.$$
Taking $\tilde \psi$ to be the $\alpha$-homogeneous extension of $\psi$, we get that  $\beta(\tilde\psi)-(1-\eps_\alpha)\beta(\psi)=0$, which concludes the proof of the lemma.
\end{proof}

\subsection{Homogeneity improvement of $\psi$}\label{sub:sing:main}
In this subsection we prove the inequality \eqref{e:sing:main}. 

We first notice that if $Q_A$ is non-negative, then we can choose $Q_B=Q_A$ and $\sum_{j=1}^k a_j^2=0$. Thus, \eqref{e:sing:main} follows directly by Lemma \ref{l:fourier} and the fact that for the eigenvalues on the sphere $\lambda_j>2d$ implies $\lambda_j\ge 3(d+1)$.

In the rest of this subsection, we assume that $Q_A$ changes sign and $Q_B$ is given by \eqref{e:Q_B}.  In particular, $Q_A-Q_B$ is a homogeneous polynomial of second degree with $\Delta Q_A-Q_B =0$, so it is an element of the eigenspace $E_{2d}$, corresponding to the eigenvalue $2d$. We choose $\phi_2\in E_{2d}$ and $c_2\in\R$ such that
$$Q_A-Q_B=c_2\phi_2\,,\quad\text{where}\quad\int_{\partial B_1}\phi_2^2(\theta)\,d\HH^{d-1}(\theta)=1.$$ 
Thus, on $\partial B_1$ we can write $\psi$ as 
$$\psi(\theta)=c_2\phi_2(\theta)+\phi(\theta)=c_2\phi_2(\theta)+\sum_{\{j\,:\, \lambda_j>2d\}} c_j \phi_j(\theta).$$
Applying Lemma \ref{l:fourier} we have 
\begin{align*}
W_0(\tilde\psi)-(1-\eps_\alpha)W_0(\psi)
&= \frac{\eps_\alpha}{d+2\alpha-2}\Big((\alpha-2)(d+\alpha)c_2^2+\sum_{\{j\,:\, \lambda_j>2d\}} (-\lambda_j+\lambda_\alpha) c_j^2\Big)\nonumber\\
&= \frac{\eps_\alpha^2(d+\alpha)^2}{(d+2\alpha-2)}c_2^2+\frac{\eps_\alpha}{d+2}\sum_{\{j\,:\, \lambda_j>2d\}} (-\lambda_j+\lambda_\alpha) c_j^2\nonumber
\end{align*}
Choosing the constant $\eps$ small enough, the equation \eqref{e:choice_of_alpha} implies that $\ds 2<\alpha\le \frac52$; 
 by the fact that $\lambda_j>2d\Rightarrow \lambda_j\ge3(d+1)$ (see Subsection \ref{sub:spectrum}) we have 
$$\lambda_j-\lambda_\alpha\ge 3(d+1)-\frac52\Big(d+\frac12\Big)> \frac{d+2}{2}\,,\quad\text{whenever}\quad \lambda_j>2d.$$
Hence, the right-hand side in the previous equality can be estimated by
\begin{align}
W_0(\tilde\psi)-(1-\eps_\alpha)W_0(\psi)
&\le \frac{\eps_\alpha^2(d+3)^2}{(d+2)}c_2^2-\frac{\eps_\alpha}2\sum_{\{j\,:\, \lambda_j>2d\}} c_j^2\nonumber\\
&= \frac{\eps_\alpha^2(d+3)^2}{(d+2)}\int_{\partial B_1}(Q_A-Q_B)^2-\frac{\eps_\alpha}2\int_{\partial B_1}\phi^2.\label{e:sing:main:part1}
\end{align}
It order to estimate the first term in the right-hand side, we notice that $Q_A-Q_B = \big(\sum_{j=1}^k a_j \big)x_d^2 - \sum_{j=1}^k a_jx_j^2$ , hence its $L^2$-norm is a degree $2$ homogeneous polynomial in $(a_1,..., a_k)$ (with coefficients depending only on $d$). Hence there exists a dimensional constant $C_d$ such that
$$\int_{\partial B_1}(Q_A-Q_B)^2 
\le C_d\sum_{j=1}^k a_j^2.$$
Together with \eqref{e:sing:main:part1}, this gives \eqref{e:sing:main}.

\subsection{The higher modes control $\sum_{j=1}^k a_j^2$}\label{sub:sing:pre-final} In this section we prove the inequality \eqref{e:sing:pre-final} from Step 3. 
Since the trace $c(\theta)$ is positive on $\partial B_1$ and can be written as
$$c(\theta)=\Big(-\sum_{j=1}^k a_j\theta_j^2+\sum_{j=k+1}^d a_j\theta_j^2\Big)+q_\nu(\theta)+\phi(\theta)\ge 0,$$
we get that 
$$\phi(\theta)\ge \Big(\sum_{j=1}^k a_j\theta_j^2-\sum_{j=k+1}^d a_j\theta_j^2\Big)_+\ge \Big(a_1\theta_1^2-\sum_{j=2}^d \theta_j^2\Big)_+ ,$$
on the half-sphere $\partial B_1\cap\{q_\nu=0\}$.
Thus, we get 
$$\phi(\theta)\ge \frac{a_1}4\theta_1^2\quad\text{on the set}\quad U_{a_1}\cap\{q_\nu=0\},\quad\text{where}\quad U_{a_1}=\Big\{\theta\in\partial B_1\ :\ a_1\theta_1^2>2\sum_{j=2}^d \theta_j^2\Big\}.$$
Notice that that for $a_1$ small enough we have 
$$\frac12(d-1)\omega_{d-1}\sqrt{a_1}^{d-1}\le \HH^{d-1}(U_{a_1})\le 2(d-1)\omega_{d-1}\sqrt{a_1}^{d-1}.$$
Thus, we obtain 
$$\int_{\partial B_1}\phi^2\ge C_d a_1^2 \sqrt{a_1}^{d-1}=C_d a_1^{(d+3)/2},$$
for a dimensional constant $C_d>0$. 
Without loss of generality we can suppose that 
$$a_1^2\ge \frac{1}{k}\sum_{j=1}^k a_j^2\ge  \frac{1}{d}\sum_{j=1}^k a_j^2,$$
and so, we get
$$\int_{\partial B_1}|\nabla_\theta \phi|^2\ge 2d\int_{\partial B_1}\phi^2\ge 2d\,C_d\, a_1^{(d+3)/2}\ge C_d\,\Big(\sum_{j=1}^k a_j^2\Big)^{(d+3)/4},$$
which gives \eqref{e:sing:pre-final} with $\ds\gamma=\frac{d-1}{d+3}$ and a dimensional constant $C_3$.
\subsection{The higher modes control $W(z)-\Theta$}\label{sub:sing:final} In this section we prove the inequality \eqref{e:sing:final} from the final Step 4.
Using the decomposition $z=q_\nu+Q_A+\psi$ and the identity \eqref{e:sing:decomposition} we get, with $\psi=r^2\,\phi$ and using $\int_{\de B_1} \phi=0$ since it contains only high modes, 
\begin{align*}
W(z)-\Theta
&=-\frac{(1-b-c_0)^2+(1-b)^2}2\Theta+\frac{1}{d+2}\int_{\partial B_1}\left(|\nabla_\theta \phi|^2-2d\phi^2\right)-\frac{c_0}{d+2}\int_{\partial B_1^+}\phi\\
&\le -\frac{c_0^2}4\Theta+\frac{1}{d+2}\int_{\partial B_1}\left(|\nabla_\theta \phi|^2-2d\phi^2\right)-\frac{c_0}{d+2}\int_{\partial B_1^+}\phi,
\end{align*}
where the last inequality follows by the fact that
$$(1-b-c_0)^2+(1-b)^2\ge \frac{c_0^2}2\,,\quad\text{for every}\quad b,c_0\in\R.$$
By the Cauchy-Schwarz inequality, we have
\begin{align*}
-\frac{c_0}{d+2}\int_{\partial B_1^+}\phi&\le \frac{c_0^2}{4}\Theta+\Big(\frac{1}{(d+2)\Theta}\int_{\partial B_1^+}\phi\Big)^2\\
&\le  \frac{c_0^2}{4}\Theta+\frac{|\partial B_1|}{(d+2)\Theta}\int_{\partial B_1}\phi^2= \frac{c_0^2}{4}\Theta+8d^2\int_{\partial B_1}\phi^2.
\end{align*}

Thus, we get 
\begin{align*}
W(z)-\Theta
&\le \frac{1}{d+2}\int_{\partial B_1}\left(|\nabla_\theta \phi|^2-2d\phi^2\right)+8d^2\int_{\partial B_1}\phi^2\\
&\le \frac{1}{d+2}\int_{\partial B_1}|\nabla_\theta \phi|^2+8d^2\int_{\partial B_1}\phi^2\le \left(\frac{1}{d+2}+4d\right)\int_{\partial B_1}|\nabla_\theta \phi|^2,
\end{align*}
where the last inequality is due to the fact that $\phi$ contains only modes $\phi_j$ corresponding to eigenvalues $\lambda_j>2d$. This gives \eqref{e:sing:final} where the constant $C_4$ can be choosen as $\ds C_4=1+4d$.

\subsection{Improvement of the decay rate}\label{sub:sing:complementary}
This subsection is dedicated to the improvement of the inequality \eqref{e:sing:pre-final}. The main result, contained in the following lemma, is more general and holds in any dimension. 
\begin{lemma}\label{l:improvement}
Suppose that $0\le k<d$ and
$$Q_A(x)=-\sum_{j=1}^k a_jx_j^2+\sum_{j=k+1}^d a_j x_j^2\ ,\quad\text{where}\quad \begin{cases} 0<a_j\quad \text{for every}\quad j=1,\dots,k,\\ 
0\le a_j\le1\quad \text{for every}\quad j=k+1,\dots,d.
\end{cases}$$
Let $\phi\in H^1(\partial B_1)$ be of zero mean, that is $\ds\int_{\partial B_1}\phi(\theta)\,d\HH^{d-1}(\theta)=0$, and such that 
$$\phi\ge Q_A\quad\text{on the half-sphere}\quad \{\xi\in \partial B_1\ :\ \xi\cdot \nu >0\},$$ 
determined by some unit vector $\nu\in\partial B_1$.

\noindent Then, there are dimensional constants $C>0$ and $\delta>0$ such that if
$\ds \sum_{j=1}^k a_j^2\le \delta$, then  
\begin{equation}
\sum_{j=1}^k a_j^2\le C_d\|\nabla_\theta\varphi\|_{L^2(\partial B_1)}^{2(1-\gamma_k)},
\end{equation}
where 
$$\gamma_k=\begin{cases} 0\,,\quad \text{if}\quad k=0,\\ 
\ds\frac{d-k}{d-k+4}\,,\quad \text{for every}\quad k=1,\dots,d-2,\\
0\,,\quad \text{if}\quad k=d-1.
\end{cases}$$
\end{lemma}

\begin{oss}
The above lemma is to be applied to the traces $c$ of the solutions $u$ of an obstacle problem, which can be written in the form $c(\theta)=Q_A(\theta)+q_\nu(\theta)+\phi(\theta)$. We notice that, although one might think that $k$ corresponds precisely to point of the $k$-th stratum, we do not know a way to deduce the precise form of $Q_A$ just from looking at the blow-up limits of $u$. This means that even if the blow up $Q_B$ is such that $\dim(\ker B)=k$, we still cannot infer anything on the structure of $B$. It follows that this result cannot be applied to improve the regularity of the singular sets of $\partial \{u>0\}$, except in dimension two, where $\gamma_0=\gamma_1=0$. This corresponds to the assumption of Weiss on the projection of $c$ on the set of admissible blow-ups $\KK$, which again finds application only in dimension two. 
\end{oss}
\begin{proof}[Proof of Lemma \ref{l:improvement}]
If $k=0$, then the inequality is trivial and so, we can suppose that $k\ge 1$.

{\it Suppose that $1\le k< d-1$. } Without loss of generality we can suppose $\ds a_1^2\ge \frac1k\sum_{i=1}^k a_i^2$.\\
Setting $X'=(x_1,\dots,x_k)$, $X''=(x_{k+1},\dots,x_d)$ and $\|\cdot\|=\|\cdot\|_{L^2(\partial B_1)}$, we have 
\begin{align*}
\Big\|\Big(\sum_{i=1}^k a_ix_i^2-\sum_{i=k+1}^da_ix_i^2\Big)_+\Big\|&\ge  \Big\|\Big(a_1x_1^2-\sum_{i=k+1}^da_ix_i^2\Big)_+\Big\|\ge \left\|\left(a_1x_1^2-|X''|^2\right)_+\right\|\\
&=\frac1k\sum_{j=1}^k\left\|\left(a_1x_j^2-|X''|^2\right)_+\right\|\ge \frac1d\left\|\left(a_1|X'|^2-|X''|^2\right)_+\right\|.
\end{align*}
We now notice that 
$$a_1|X'|^2-|X''|^2\ge \frac{a_1}2|X'|^2\quad\text{on the set}\quad U_{a_1}=\Big\{X=(X',X'')\in\partial B_1\ :\ \frac{a_1}2|X'|^2\ge  |X''|^2\Big\},$$
and for $a_1$ small enough we get
$$\frac12k\omega_k\sqrt{a_1}^{d-k}\le \HH^{d-1}(U_{a_1})\le 2k\omega_{k}\sqrt{a_1}^{d-k}.$$
In particular, 
$$\|\inf\{Q_A,0\}\|_{L^2(\partial B_1)}^2\ge \left\|\left(a_1|X'|^2-|X''|^2\right)_+\right\|_{L^2(\partial B_1)}^2\ge  C_da_1^{\frac{d-k+4}2}\ge C_d \Big(\sum_{i=1}^k a_i^2\Big)^{\frac{d-k+4}4}.$$
Now, since $Q_A$ is even and $\|\phi\|_{L^2(\partial B_1)}^2\le \frac{1}{d-1}\|\nabla_\theta\phi\|_{L^2(\partial B_1)}^2$, we obtain the claimed inequality
$$\sum_{i=1}^k a_i^2\le C_d\|\nabla_\theta\phi\|_{L^2(\partial B_1)}^{\frac8{d-k+4}}.$$

{\it Suppose that $k=d-1$. } We argue by contradiction. Suppose that there are a sequence of functions $\phi_n:\partial B_1\to\R$ of zero mean and vectors $\nu_n\in\partial B_1$ and $(a_n^1,\dots,a_n^k)$ such that 
$$\phi_n(\theta)\ge a_n^1\theta_1^2\quad\text{on the set}\quad \{\theta\in\partial B_1\ : \theta_d=0\,,\ \theta\cdot \nu_n>0\},$$
$$\ds a_n^1\ge \Big(\frac1k\sum_{j=1}^k|a_n^j|^2\Big)^{1/2}\qquad\text{and}\qquad  \sum_{j=1}^k|a_n^j|^2\ge n\|\nabla\phi_n\|_{L^2(\partial B_1)}^2.
$$
Thus, the sequence of functions $\psi_n:=\phi_n/ a_n^1$ is such that $\ds\lim_{n\to\infty }\|\nabla\psi_n\|_{L^2(\partial B_1)}^2=0$ and 
$$\psi_n(\theta)\ge \theta_1^2\quad\text{on the set}\quad \{\theta\in\partial B_1\ : \theta_d=0\,,\ \theta\cdot \nu_n>0\},$$
which is in contradiction with the trace inequality 
$$\int_{\{\theta_d=0\}\cap \partial B_1}\psi_n^2\,d\HH^{d-2}\le C\int_{\partial B_1}\big(|\nabla\psi_n|^2+\psi_n^2\big)\,d\HH^{d-1}.$$
\end{proof}

\subsection{On the sharpness of the non-homogeneous estimate in Theorem \ref{t:epi:B}} We conclude this section with an example, which 
shows that in dimension higher than three one cannot estimate the distance to the cone $\mathcal K$ by just using the energy of the higher modes $\phi$ to the power one. Indeed, such an estimate would be in contradiction with inequality \eqref{e:esempio} below. In particular, Example \ref{esempio} shows that for general traces in higher dimension our method cannot be improved.
\begin{exam}\label{esempio}
Consider the non-negative trace $c:\partial B_1\to\R^+$ given by
$$c(\theta)=\left(\frac1{4(d-1)}\sum_{j=1}^{d-1}\theta_j^2 - \eps\theta_d^2\right)_+.$$
Notice that, since $c$ is even its Fourier expansion on the sphere $\partial B_1$ does not contain linear terms. As in the proof of Theorem \ref{t:epi:B}, the trace $c$ can be uniquely decomposed as $c(\theta)=Q(\theta)+\phi(\theta),$ where $Q$ is a homogeneous polynomial of second degree and $\phi$ contains only higher modes, that is 
$$\ds \phi(\theta)=\sum_{\{j\,:\,\lambda_j>2d\}}c_j\phi_j(\theta).$$
We claim that 
\begin{equation}\label{e:esempio}
\|\nabla_\theta \phi\|_{L^2(\partial B_1)}^{\frac4{d+1}}\lesssim \text{dist}_{L^2(\partial B_1)}(Q,\mathcal K).
\end{equation}
In order to prove \eqref{e:esempio} we set 
$$P(\theta)= \frac1{4(d-1)}\sum_{j=1}^{d-1}\theta_j^2- \eps\theta_d^2\qquad\text{and}\qquad R(\theta)=\left(\eps\theta_d^2-\frac1{4(d-1)}\sum_{j=1}^{d-1}\theta_j^2\right)_+,$$
and we notice that $c(\theta)=P(\theta)-R(\theta).$
It is easy to check that the term $R$ has the following asymptotic behavior when the parameter $\eps$ is small: 
$$\|R\|_{L^\infty(\partial B_1)}=\eps\ ,\qquad  \HH^{d-1}(\{R>0\})\sim {\eps}^{\frac{d-1}2}\ , \qquad\|\nabla_\theta R\|_{L^\infty(\partial B_1)}\sim\sqrt\eps\ ,$$
$$\|R\|_{L^2(\partial B_1)}\sim  {\eps}^{\frac{d+3}4}\qquad\text{and}\qquad \|\nabla_\theta R\|_{L^2(\partial B_1)}\sim {\eps}^{\frac{d+1}4}.$$
The function $R$ can be decomposed as 
$$R(\theta)=\frac{c_0}{\sqrt{\HH^{d-1}(\partial B_1)}}+c_2\phi_2(\theta)-\phi(\theta),$$
where
\begin{itemize}
\item $c_0\in\R$ corresponds to the first (constant) mode of the Fourier expansion of $R$ on $\partial B_1$ and can be estimated in terms of $\eps$ as
$$\ds c_0=\frac{1}{\sqrt{\HH^{d-1}(\partial B_1)}}\int_{\partial B_1}R\,d\HH^{d-1}\le  \|R\|_{L^2(\partial B_1)}\lesssim {\eps}^{\frac{d+3}4};$$ 
\item $\phi_2(\theta)$ is an eigenfunction of the Laplacian on the sphere corresponding to the eigenvalue $2d$ and $\|\phi_2\|_{L^2(\partial B_1)}=1$ and the constant $c_2\in\R$ can be estimated as 
$$\ds |c_2|\le \left|\int_{\partial B_1}R\phi_2\right|\le \|R\|_{L^2(\partial B_1)}\lesssim {\eps}^{\frac{d+3}4};$$
\item the function $\phi$ is precisely the one from the decomposition of $c$, contains only higher modes and satisfies the following estimate: 
$$\ds \|\nabla_\theta \phi\|_{L^2(\partial B_1)}\le \|\nabla_\theta R\|_{L^2(\partial B_1)}+|c_2|\|\nabla_\theta \phi_2\|_{L^2(\partial B_1)}\lesssim {\eps}^{\frac{d+1}4}+{\eps}^{\frac{d+3}4}2d\lesssim {\eps}^{\frac{d+1}4}.$$
\end{itemize}
On the other hand, the $L^2(\partial B_1)$ distance from $Q=P-c_0-c_2\phi_2$ to the cone $\mathcal K$ of nonnegative homogeneous polynomials of second degree has the behavior 
$$\text{dist}_{L^2(\partial B_1)}\Big(P-c_0-c_2\phi_2,\mathcal K\Big)\sim \text{dist}_{L^2(\partial B_1)}\big(P,\mathcal K\big)\sim \eps. $$
Thus, we finally get the claimed inequality \eqref{e:esempio}
$$  \|\nabla_\theta \phi\|_{L^2(\partial B_1)}^{\frac4{d+1}}\lesssim \eps\sim \text{dist}_{L^2(\partial B_1)}\Big(P-c_0-c_2\phi_2,\mathcal K\Big).$$ 
\end{exam}

%
%
%

\section{Uniqueness of blow-up and regularity of free boundary}\label{s:regularity} 
In this Section we prove Theorems~\ref{t:uniq} and~\ref{t:reg}, focusing on the statement 2 of each result. We show in detail how the logarithmic estimates follow from the ``modified'' epiperimetric inequality of Theorem~\ref{t:epi:B} and we prefer to skip the analogous estimates on the H\"older continuity at regular points, since this is the main improvement of the present paper and since the proof of the latter is a simpler version of the estimates below and it is already contained in \cite[Theorem 4 and 5]{Weiss2}.

\begin{prop}\label{p:decay}
	Let $\Omega\subset\R^d$ be an open set and $u\in H^1(\Omega)$ a minimizer of $\mathcal E$. Then for every compact set $K\Subset\Omega$, there is a constant $C:=C(d, K, \Omega)>0$ such that for every free boundary point $x_0\in {\rm Sing }(u)\cap K$, 
        the following decay holds
	\begin{equation}\label{e:L^2decay}
	\| u_{x_0,t}-u_{x_0,s} \|_{L^1(\partial B_1)} \leq C\, (-\log(t))^{-\frac{1-\gamma}{2\gamma}} \qquad  \mbox{for all}\quad 0< s<t < \dist (K,\de \Omega)\,.
	\end{equation} 
\end{prop}

\begin{proof}
{\it Step 1 (closeness of the blow ups for a given point $x_0$).} Let 
$x_0 \in K$ and let $r_0\in (0,  \dist (K,\de \Omega)]$ be such that the epiperimetric inequality of Theorem~\ref{t:epi:B} can be applied to the rescaling $u_{x_0,r}$ for every $r \leq r_0$. We claim that
$$	\| u_{x_0,t}-u_{x_0,s} \|_{L^1(\partial B_1)} \leq C\, (-\log(t/r_0))^{-\frac{1-\gamma}{2\gamma}} \qquad  \mbox{for all}\quad 0< s<t < r_0\,.
$$

We assume $x_0=0$ without loss of generality and 
$$e(r) = W(u,r)-\Theta_u(0).
$$
By the monotonicity formula \eqref{e:Weiss_monotonicity} and the epiperimetric inequality of Theorem~\ref{t:epi:B}, there exists a radius $r_0>0$ such that for every $r \leq r_0$
\begin{equation}
\label{eqn:epi-applied}
\frac{d}{dr}e(r) \geq \frac{d+2}{r} \big( W(c_r)-\Theta_u(0)-e(r) \big) +f(r)\geq \frac{c}{r} e(r)^{1+\gamma} +2f(r)
\end{equation}
where $\gamma \in (0,1)$ is a dimensional constant and
$$f(r):=\frac{1}{r}\int_{\partial B_1}|x\cdot \nabla u_r-2u_r|^2\,d\HH^{1}.$$
We obtain that
\begin{equation}
\label{e:mon_rem}\frac{d}{dr}\Big(\frac{-1}{\gamma e(r)^\gamma} - c \log r\Big) =\frac{1}{ e(r)^{1+\gamma}}\frac{d}{dr}e(r)- \frac{c}{r} \geq \frac{1}{ e(r)^{1+\gamma}} f(r) \geq 0
\end{equation}
and this in turn implies that $-{e(r)^{-\gamma}} - c\gamma \log r$ is an increasing function of $r$, namely that $e(r)$ decays as 
\begin{equation}
\label{eqn:e-logar}
e(r) \leq ({e(r_0)^{-\gamma}+c \gamma \log r_0-c \gamma \log r})^{\frac {-1}{\gamma}} \leq (-c \gamma \log (r/r_0))^{\frac {-1}{\gamma}}.
\end{equation}
For any $0<s<t<r_0$ we estimate the $L^1$ distance between the blow ups at scale $s$ and $t$ through the Cauchy-Schwarz inequality and the monotonicity formula \eqref{e:Weiss_monotonicity} 
\begin{align}\label{e:imp_1}
\int_{\de B_1}\left| u_t-u_s \right| \,d\,\HH^{n-1}
&\leq \int_{\de B_1}\int_s^t\frac{1}{r}\left|  x\cdot \nabla u_r-2u_r   \right| \,dr\,d\HH^{n-1} \notag\\
&\leq \big({n \omega_n}\big)^{1/2} \int_s^t \frac 1 r \left(\frac 1 r \int_{\de B_1} \left|  x\cdot \nabla u_r-2u_r   \right|^2   \,d\HH^{n-1} \right)^{1/2}\,dr \notag\\
&\leq \Big({ \frac{n \omega_n}{2}} \Big)^{1/2} \int_s^t \frac 1 r (e'(r))^{1/2}\, dr \notag\\
& \leq \Big({ \frac{n \omega_n}{2}} \Big)^{1/2} (\log(t)-\log(s))^{1/2} (e(t)-e(s))^{1/2}
 \,.
\end{align}
Let $0<s^2<t^2<r_0$ such that $s/r_0\in [2^{-2^{i+1}}, 2^{-2^i})$, $t/r_0\in [2^{-2^{j+1}}, 2^{-2^j})$ for some $j\leq i$ and applying the previous estimate\eqref{eqn:e-logar} to the exponentially dyadic decomposition, we obtain 
\begin{align}\label{e:imp_2}
\int_{\de B_1}\left| u_t-u_s \right| \,d\,\HH^{n-1}&\leq \int_{\de B_1}\left| u_t- u_{2^{-2^{j+1}}r_0} \right| \,d\,\HH^{n-1}  \notag
\\&+ \int_{\de B_1}\left| u_{2^{-2^{i}}r_0}-u_s \right| \,d\,\HH^{n-1} +  \sum_{k=j+1}^{i-1}  \int_{\de B_1}\left| u_{2^{-2^{k+1}}r_0}-u_{2^{-2^{k}}r_0} \right| \,d\,\HH^{n-1} \notag
\\
&\leq C \sum_{k=j}^{i} \left(\log\big(2^{-2^{k}}\big)- \log\big(2^{-2^{k+1}}\big)\right)^{1/2}\left(e\big(2^{-2^{k}}\big)- e\big(2^{-2^{k+1}}\big) \right)^{1/2}\notag
\\
&\leq C \sum_{k=j}^{i} 2^{k/2}e\big(2^{-2^{k}}\big)^{1/2}
\leq C \sum_{k=j}^{i} 2^{(1-1/\gamma)k/2} 
\\&\leq C 2^{(1-1/\gamma)j/2} \leq C (-\log(t/r_0))^{\frac{\gamma-1}{2\gamma}}\notag
  \,,
\end{align}
where $C$ is a dimensional constant that may vary from line to line.

\smallskip
{\it Step 2 (uniform smallness of monotonic quantity for $x_0 \in {\rm Sing }(u)\cap  K$).} We claim that for every $\eps>0$ there exists $r_0>0$ such that
$$e(u_{x,r}) \leq \eps \qquad \mbox{for every }x\in  {\rm Sing }(u)\cap K,\; r\leq r_0.$$

Assume by contradiction that there exists a sequence $x_k \to x_0$ and $r_k \to 0$ such that 
$\eps <e(u_{x_k,r_k})$ for any $k\in \N$.
By the monotonicity of $W$, for any $\rho>0$ and $k$ large enough 
$$\eps <W(u,{x_k,r_k})- \Theta_u(0) \leq W(u,{x_k,\rho})- W(u,{x_0,\rho})+W(u,{x_0,\rho})- \Theta_u(0).$$
In turn, the right-hand side can be made arbitrarily small by choosing first $\rho$ sufficiently small (to make the difference of the last two terms small) and then $k$ sufficiently large.

\smallskip
{\it Step 3 (uniform scale for the application of the epiperimetric inequality at $x_0 \in {\rm Sing }(u)\cap  K$).} 
We claim that for every $\eps>0$ there exists $r_0>0$ such that
$$\dist_{L^2}(u_{x,r}, K ) \leq \eps \qquad\mbox{for every }x\in  {\rm Sing }(u)\cap K,\; r\leq r_0.
\qquad 
$$
(notice that this statement holds also if in place of the $L^2$-distance we consider the $H^1$-distance).

Assume by contradiction that there exists $\eps>0$ a sequence $x_k \to x_0$ and $r_k \to 0$ such that 
\begin{equation}
\label{eqn:faraway}
\eps < \dist_{L^2}(u_{x_k,r_k}, K )  \qquad \mbox{for any }k\in \N.
\end{equation}

Since the sequence $\{u_{x_k,r_k}\}_{k\in \N}$ is uniformly bounded in $H^{2,\infty}$, it converges strongly in $H^1$ up to a (not relabelled) subsequence to $u_0$. Moreover, thanks to Step 2, the limit $u_0$ must satisfy $W(u_0, x_0, 1)=  \Theta_u(0)$, so that it belongs to $K$. This contradicts \eqref{eqn:faraway}.

\smallskip
{\it Step 4 (conclusion).} We can now conclude the proof of the Proposition.

We observe that for every $r_0>0$ and $t \leq r_0^2$, we have $\log(t/r_0) \leq 2 \log t$. From Step 1 and 3, we deduce that there exists $r_0>0$ such that for all $0< s<t < r_0^2$, $x_0\in  {\rm Sing }(u)\cap K$
$$	\| u_{x_0,t}-u_{x_0,s} \|_{L^1(\partial B_1)} \leq C\, (-\log(t))^{-\frac{1-\gamma}{2\gamma}}.  
$$
From \eqref{e:imp_1} we have
$$	\| u_{x_0,t}-u_{x_0,r_0^2} \|_{L^1(\partial B_1)} \leq C (-\log(r_0))^{1/2} e(\dist (K,\de \Omega))^{1/2}
$$
and the right hand side is estimated by $C (-\log(r_0))^{-\frac{1-\gamma}{2\gamma}}$ for a constant $C$ depending only on $d, r_0$, $ e(\dist (K,\de \Omega))$, $\dist (K,\de \Omega)$.

\end{proof}




As a consequence of the previous proposition we can prove the uniqueness of the blow up Theorem \ref{t:uniq}, with a logarithmic rate of convergence of the blow up sequence at each point of the singular set (and uniform in any compact set inside the domain).

\begin{proof}[Proof of Theorem~\ref{t:uniq}]
We notice that
$$|Q_{x_1} - Q_{x_2}| \leq c(n) \int_{\partial B_1} | Q_{x_1}(x) - Q_{x_2}(x)| \, d\HH^{n-1}(x)$$
By the triangular inequality
$$\| Q_{x_1} - Q_{x_2}\|_{L^1(\partial B_1)} \leq \| u_{x_1, r} - Q_{x_1}\|_{L^1(\partial B_1)} + \| u_{x_1, r}-u_{x_2, r}\|_{L^1(\partial B_1)} + \| u_{x_2, r} - Q_{x_2}\|_{L^1(\partial B_1)}  
$$

Recalling that $u \in C^{1,1}$ and that $\nabla u(x_1)= 0$, we estimate the term in the middle with
\begin{equation}
\begin{split}
\| u_{x_1, r}-u_{x_2, r}\|_{L^1(\partial B_1)} &\leq \int_{\partial B_1} \int_0^1  \frac{ |\nabla u (x_1+rx+ t(x_2-x_1))| |x_2-x_1|}{r^2} \, dt \, d\HH(x)
\\
&\leq
C\| u \|_{C^{1,1}(B_r(x_0))} \frac{(r+|x_2-x_1|)\, |x_2-x_1|}{r^2}
\end{split}
\end{equation}

We choose $r= |x_1-x_2| (-\log |x_1-x_2|)^{-\frac{1-\gamma}{2\gamma}}$ and we assume that $r_0$ satisfies the inequality $ |r_0| (-\log |r_0|)^{-\frac{1-\gamma}{2\gamma}} \leq \dist (K , \partial \Omega)$. By Theorem~\ref{t:uniq} we see that
\begin{equation}
\begin{split}
\| u_{x_1, r} - Q_{x_1}\|_{L^1(\partial B_1)} + \| u_{x_2, r} - Q_{x_2}\|_{L^1(\partial B_1)}  
& \leq C (-\log(r ))^{-\frac{1-\gamma}{2\gamma}} 
\\&=  C (-\log |x_1-x_2| - \frac{1-\gamma}{2\gamma} \log( \log |x_1-x_2|))^{-\frac{1-\gamma}{2\gamma}}
\end{split}
\end{equation}
Noticing that the inequality $a- \frac{1-\gamma}{2\gamma} \log a \geq a/2$ holds for $a$ greater than a given $a_0 >0$ (depending only on $\gamma$ and therefore on $d$), we apply this inequality to $a= -\log |x_1-x_2|$ to get
$$ \| u_{x_1, r} - Q_{x_1}\|_{L^1(\partial B_1)} + \| u_{x_2, r} - Q_{x_2}\|_{L^1(\partial B_1)}  
 \leq C (-\log|x_1-x_2| )^{-\frac{1-\gamma}{2\gamma}}.
 $$

Putting together the previous inequalities, we find \eqref{eqn:log-cont-fb}.
\end{proof}

\subsection{Proof of Theorem \ref{c:Q_funct}} We notice that if $u\in H^1(\Omega)$ is a minimizer of $\mathcal E_q$ or $\mathcal A_q$, then it is locally $W^{2,\infty}$ by the results of \cite{Gerh} and moreover it is an almost-minimizer of the functional $\mathcal E$ with a constant $C$ depending only on $\|q\|_{C^{0,\gamma}(\Omega)},c_q$ and $\| u\|_{W^{2,\infty}_{loc}}$.

We say that $u\in H^1(\Omega)$ is an almost minimizer of $\mathcal E$ if there exists a constant $C>0$ such that for every ball $B_{r}(x_0)\subset \Omega$ and for every $v \in H^1(B_r(x_0))$ which agrees with $u$ on $\partial B_r(x_0)$ 
\begin{equation}\label{e:alm_mon}
\int_{B_r(x_0)}\left[|\nabla u|^2+q(x_0)\max\{u, 0\} \right]\,dx\leq (1+ Cr^\gamma)\int_{B_r(x_0)}\left[|\nabla v|^2+q(x_0)\max\{v, 0\}\right]\,dx\,.
\end{equation}

In the following we show that the statement of Theorem \ref{c:Q_funct}, in particular the logarithmic estimate, holds true also if we drop the assumption that $u \in W^{2,\infty}_{loc}(\Omega)$ is a minimizer of $\mathcal E_q$ or $\mathcal A_q$ and we only assume the almost minimality.

The main modifications with respect to the arguments of Section~\ref{s:regularity} appear in Proposition \ref{p:decay} and we outline them below.

Up to a rescaling, we may assume that $q(x_0)=1$. Applying the epiperimetric inequality of Theorem~\ref{t:epi:B}  to $u_{r,x_0}$, we find that \eqref{eqn:epi-applied} has to be modified for almost monotonicity \eqref{e:alm_mon} to get
$$
\frac{d}{dr}e(r) \geq \frac{n+2}{r} \big( W(c_r)-\Theta_u(0)-e(r)  \big) +f(r)\geq \frac{c_0}{r} e(r)^{1+\gamma} - \frac{c_1}{r^{1-\alpha}} +2f(r)
$$
(where $e(r):=W(u_r,1)-\Theta_u(0)$ and the notation is the same as in Section~\ref{s:regularity}) for some constants $c_0,c_1>0$. 

We define now $\tilde e(r)= e(r)+2 \alpha^{-1} c_1 r^\alpha$ and we notice that from the previous inequality and since $a^{1+\gamma}+b^{1+\gamma} \geq 2^{-\gamma} (a+b)^{1+\gamma}$ for any $a,b \geq 0$
\begin{equation*}
\tilde e'(r) \geq \frac{c_0}{r} e(r)^{1+\gamma} + \frac{c_1}{r^{1-\alpha}} +2f(r)
\geq \frac{c_0}{r} [e(r) + {c_1}{r^{\frac{\alpha}{1+\gamma}}} ]^{1+\gamma}+2f(r)
\end{equation*}
For $r$ sufficiently small, the previous inequality implies that 
$$\tilde e'(r) \geq \frac{c_0}{r} \tilde e(r)^{1+\gamma}+2f(r)
$$
From the previous inequality, we see that $\tilde e(r)$ satisfies the same inequality that $e(r)$ solved in \eqref{eqn:epi-applied}. Hence, with the same argument as in \eqref{e:mon_rem}, we see that $\tilde e$ satisfies the same estimate as $e$ in \eqref{eqn:e-logar}
$$e(r)+2 \alpha^{-1} c_1 r^\alpha = \tilde e(r) \leq (-c \gamma \log (r/r_0))^{\frac {-1}{\gamma}}.
$$
This inequality implies that, up to a constant, also $e(r)$ satisfies a logarithmic estimate and we can carry out the rest of the proof of Proposition~\ref{p:decay} and of Theorems \ref{t:uniq} and \ref{t:reg}.
\qed

\hfill

\bibliographystyle{plain}
\bibliography{references-Cal}

\end{document}